\documentclass[12pt]{article}

\usepackage{floatrow}%
\usepackage{graphicx}
\usepackage{caption}
\usepackage{subcaption}
\usepackage{amsmath,amssymb}
\usepackage{tikz}
\usepackage{color}
\usepackage[toc]{appendix}
\usepackage{graphicx}
\usepackage{fancyhdr}
\usepackage{enumitem}
\usepackage{bbm}
\usepackage{graphicx}
\usepackage{parskip}
\usepackage{float}
\usepackage{chngpage}
\usepackage{calc}
\usepackage{bigints}
\usepackage{tikz}
\usepackage{array}
\usepackage{booktabs}
\usepackage{rotating}
\usepackage{multirow}
\usepackage{adjustbox}
\usepackage{tabularx}
\usepackage{verbatim}
\usepackage{mathtools}
\usepackage{ragged2e}
\usepackage[makeroom]{cancel}
\usepackage{enumitem}
\usepackage{caption}
\usepackage{hyperref}
\usepackage{pdfpages}

\usepackage{amsmath,amssymb}
\usepackage{appendix}
\usepackage{graphicx}
\usepackage{pgfplots}
\usepackage{tikz}
\usepackage{verbatim}
\usepackage{amsthm}
\usepackage[english]{babel}
\usepackage{hyphenat}
\usepackage[makeindex]{imakeidx}
\usepackage{tikz}
\usetikzlibrary{datavisualization}
\usetikzlibrary{matrix}
\usetikzlibrary{datavisualization.formats.functions}

\newtheorem{theorem}{Theorem}[section]
\newtheorem{proposition}[theorem]{Proposition}

\newtheorem{remark}[theorem]{Remark}

\makeatletter
\def\section{\@startsection {section}{1}{\z@}{3.25ex plus 1ex minus
		.2ex}{1.5ex plus .2ex}{\large\bf}}
\def\subsection{\@startsection{subsection}{2}{\z@}{3.25ex plus 1ex minus
		.2ex}{1.5ex plus .2ex}{\normalsize\bf}}
\@addtoreset{equation}{section} 
\makeatother

\title{Upper and lower bounds for the solution of a stochastic prey-predator system with foraging arena scheme}

\author{Alberto Lanconelli\thanks{Dipartimento di Scienze Statistiche Paolo Fortunati, Università di Bologna, Bologna, Italy. \textbf{e-mail}: alberto.lanconelli2@unibo.it} \and  Ramiro Scorolli\thanks{Dipartimento di Scienze Statistiche Paolo Fortunati, Università di Bologna, Bologna, Italy. \textbf{e-mail}: ramiro.scorolli2@unibo.it}}
\date{\today}

\begin{document}
	
\maketitle
	
\bigskip
	
\begin{abstract}
We investigate some probabilistic aspects of the unique global strong solution of a two dimensional system of stochastic differential equations describing a prey-predator model perturbed by Gaussian noise. We first establish, for any fixed $t> 0$, almost sure upper and lower bounds for the components $X(t)$ and $Y(t)$ of the solution vector: these explicit estimates emphasize the interplay between the various parameters of the model and agree with the asymptotic results found in the literature. Then, standing on the aforementioned bounds, we derive upper and lower estimates for the joint moments and distribution function of $(X(t),Y(t))$.
Our analysis is based on a careful use of comparison theorems for stochastic differential equations and exploits several peculiar features of the noise driving the equation.    	
\end{abstract}
	
	
	
\allowdisplaybreaks

\section{Introduction}

In theoretical ecology the system of equations
\begin{align}
\begin{cases}\label{Mao system general}
\frac{dx(t)}{dt}=x(t)(a_1-b_1x(t))-c_1h(x(t),y(t))y(t),& x(0)=x;\\
\frac{dy(t)}{dt}=y(t)(-a_2-b_2y(t))+c_2 h(x(t),y(t))y(t),& y(0)=y,
\end{cases}
\end{align}
constitutes a fundamental class of models for predator-prey interaction. Here, $x(t)$ and $y(t)$ represent the population densities of prey and predator at time $t\geq 0$, respectively; $a_1$ the prey intrinsic growth rate; $a_2$ the predator intrinsic death rate; $a_1/b_1$ the carrying capacity of the ecosystem; $b_2$ the predator intraspecies competition; $h(x(t),y(t))$ the intake rate of predator; $c_2/c_1$ the trophic efficiency. We observe that equation (\ref{Mao system general}) encompasses the classic Lotka-Volterra model \cite{Lotka},\cite{Volterra} which is obtained setting $b_1=b_2=0$ and $h(x,y)=x$. \\
To catch the different features of specific environments, several choices for the so-called functional response $h(x,y)$ have been suggested in the literature; we mention, among others,
\begin{itemize}
	\item Holling II function \cite{Holling}: $h(x,y)=\frac{x}{\beta+x}$; 
	\item ratio dependent functional responses \cite{Akcakaya},\cite{Arditi}: $h(x,y)=\tilde{h}(x/y)$;
	\item foraging arena models \cite{Ahrens},\cite{Walters}: $h(x,y)=\frac{x}{\beta+\alpha_2y}$;
	\item Beddington-DeAngelis model \cite{Beddington},\cite{DeAngelis}: $h(x,y)=\frac{x}{\beta+\alpha_1 x+\alpha_2 y}$;
	\item Crowley‐Martin model \cite{Crowley}: $h(x,y)=\frac{x}{\beta+\alpha_1 x+\alpha_2 y+\alpha_3 xy}$;
	\item Hassell‐Varley model \cite{Sutherland}: $h(x,y)=\frac{x}{\alpha_1 x+\alpha_2y^m}$.
\end{itemize}
($\beta,\alpha_1,\alpha_2,\alpha_3,$ are positive real numbers, $m\in\mathbb{N}$ and $\tilde{h}:\mathbb{R}\to\mathbb{R}$ a suitable regular function). What distinguishes the Holling II function from other models is the absence of $y$; on this issue the paper \cite{Skalki} presents statistical evidence from 19 predator–prey systems that the Beddington‐DeAngelis, Crowley‐Martin and Hassell‐Varley models (whose functional responses depend on both prey and predator abundances) can provide better descriptions compared to those with Holling-type functions (see also \cite{Heath}). Moreover, as remarked in \cite{Abrams}, models based on ratio-dependent functional responses exhibit singular behaviours.

With the aim of introducing environmental noise in the model, different types of stochastic perturbation for the system (\ref{Mao system general}) have been considered and studied. Among the most common, we find the It\^o-type stochastic differential equation
  \begin{align}
  \begin{cases}\label{stochasticx Mao system general}
  dX(t)=\left[X(t)(a_1-b_1X(t))-c_1h(X(t),Y(t))Y(t)\right]dt+\sigma_1X(t)dB_1(t),& X(0)=x;\\
  dY(t)=\left[Y(t)(-a_2-b_2Y(t))+c_2h(X(t),Y(t))Y(t)\right]dt+\sigma_2Y(t)dB_2(t),& Y(0)=y,
  \end{cases}
  \end{align}
where $\{(B_1(t),B_2(t))\}_{t\geq 0}$ is a standard two dimensional Brownian motion and $\sigma_1,\sigma_2$ positive real numbers. System (\ref{stochasticx Mao system general}) tries to catch random fluctuations in the growth rate $a_1$ and death rate $a_2$. Some references in this stream of research are \cite{Mao foraging}, in the case of foraging arena schemes, \cite{Du}, \cite{Ji}, \cite{Li} treating the case of Beddington-DeAngelis functional response, and \cite{Rao} dealing with Hassell-Varley model. It is worth mentioning that all these papers are devoted to the study of global existence, uniqueness, positivity and asymptotic properties for the specific model of type (\ref{stochasticx Mao system general}) considered.  

Our investigation is focused on the system    
	\begin{align}
	\begin{cases}\label{Mao system}
	dX(t)=\left[X(t)(a_1-b_1X(t))-c_1\frac{X(t)Y(t)}{\beta+Y(t)}\right]dt+\sigma_1X(t)dB_1(t),& X(0)=x;\\
	dY(t)=\left[Y(t)(-a_2-b_2Y(t))+c_2\frac{X(t)Y(t)}{\beta+Y(t)}\right]dt+\sigma_2Y(t)dB_2(t),& Y(0)=y,
	\end{cases}
	\end{align}
which is proposed and analysed in \cite{Mao foraging}. It corresponds to equation (\ref{stochasticx Mao system general}) with a foraging arena functional response. It is proved in \cite{Mao foraging} that system (\ref{Mao system}) possesses a unique global strong solution $\{(X(t),Y(t))\}_{t\geq 0}$ fulfilling the condition
\begin{align*}
\mathbb{P}(X(t)>0\mbox{ and }Y(t)>0,\mbox{ for all $t\geq 0$})=1.
\end{align*}
Moreover, the authors investigate the asymptotic behaviours of $X(t)$ and $Y(t)$, as $t$ tends to infinity, and identify three different regimes: 
\begin{itemize}
\item if $a_1<\frac{\sigma_1^2}{2}$, then  
\begin{align}\label{Mao asymp 1}
\lim_{t\to+\infty}X(t)=\lim_{t\to+\infty}Y(t)=0,
\end{align}
almost surely and exponentially fast;
\item if $\frac{\sigma_1^2}{2}<a_1<\frac{\sigma_1^2}{2}+\frac{b_1\beta a_2}{c_2}+\frac{b_1\beta \sigma_2^2}{2c_2}=:\phi$, then almost surely 
\begin{align}\label{Mao asymp 2}
\lim_{t\to+\infty}Y(t)=0,\quad\mbox{ exponentially fast}, 
\end{align}
and
\begin{align}\label{Mao asymp 2 bis}
\lim_{t\to+\infty}\frac{1}{t}\int_0^tX(r)dr=\frac{a_1-\sigma_1^2/2}{b_1};
\end{align}
\item if $a_1>\frac{\phi}{1-\sigma_2^2/2c_2-a_2/c_2}$ and $a_2+\frac{\sigma_2^2}{2}<c_2$, then system (\ref{Mao system}) has a unique stationary distribution.
\end{itemize} 
The case 
\begin{align*}
\phi<a_1<\frac{\phi}{1-\sigma_2^2/2c_2-a_2/c_2},
\end{align*}
with $a_2+\frac{\sigma_2^2}{2}<c_2$, is not investigated but the authors mention that computer simulations indicate the existence of stationary distributions for both $X(t)$ and $Y(t)$ also in that regime. 

The goal of our work is to present a novel analysis for systems of the type (\ref{stochasticx Mao system general}), which in the current study take the form (\ref{Mao system}). We derive explicit upper and lower bounds for the components $X(t)$ and $Y(t)$ of the solution of equation (\ref{Mao system}) at any fixed time $t\geq 0$. Such almost sure estimates depend solely on the parameters describing the model under investigation and the noise driving the equation. Their derivation is based on a careful use of comparison theorems for stochastic differential equations and standard stochastic calculus' tools. The estimates we obtain reflect the intrinsic interplay between the parameters of the model and enlighten the probabilistic dependence structure of $X(t)$ and $Y(t)$. We also remark that our bounds, which are valid for any fixed time $t\geq 0$, agree in the limit as $t$ tends to infinity with the asymptotic results proven in \cite{Mao foraging} and summarized above. We then utilize the previously mentioned bounds to get upper and lower estimates for the joint moments and distribution function of $(X(t),Y(t))$. We propose closed form expressions which rely on new estimates for a logistic-type stochastic differential equation.\\
It is important to remark that, while systems of the type (\ref{stochasticx Mao system general}) with Beddington-DeAngelis or Crowley-Martin or Hassell-Varley functional responses can be treated, as far as finite time analysis is concerned, with a change of measure approach, the unboundedness of $h(x,y)=\frac{x}{\beta+\alpha_2y}$, as a function of $x$, prevents from the use of a similar approach for (\ref{Mao system}). We will in fact prove in Section 3.1 below the failure of the Novikov condition for the corresponding change of measure.
    
The paper is organized as follows: Section 2 collects some auxiliary results on the solution of a logistic stochastic differential equation that plays a major role in our analysis; in Section 3 we state and prove our first main theorem: almost sure upper and lower bounds for $X(t)$ and $Y(t)$, for any $t\geq 0$. Here, we also comment on the impossibility of a change of measure approach and compare our findings with the asymptotic results from \cite{Mao foraging}; Section 4 contains our second main result, which proposes upper and lower estimates for the joint moments of $(X(t),Y(t))$; in Section 5 upper and lower bounds for the joint probability function of $(X(t),Y(t))$ constitutes our third and last main theorem; the last section contains a discussion of the result obtained in the paper and some numerical simulations of the proposed bounds.  
 
\section{Preliminary results}

In this section we will prove some auxiliary results concerning the solution of the logistic stochastic differential equation 
\begin{align}\label{logistic}
dL(t)=L(t)(a-bL(t))dt+\sigma L(t)dB(t),\quad L(0)=\lambda.
\end{align}
Here $a$, $b$, $\sigma$ and $\lambda$ are positive real numbers and $\{B(t)\}_{t\geq 0}$ is a standard one dimensional Brownian motion. It is well known (see for instance formula (4.51) in \cite{Kloeden Platen} or formula (2.1) in \cite{Jiang} for the case of time-dependent parameters) that equation (\ref{logistic}) possesses a unique global positive strong solution which can be represented as
\begin{align}\label{explicit logistic}
L(t)=\frac{\lambda e^{(a-\sigma^2/2)t+\sigma B(t)}}{1+b\int_0^t\lambda e^{(a-\sigma^2/2)r+\sigma B(r)}dr}, \quad t\geq 0.
\end{align}  
We start focusing on the asymptotic behaviour of the solution of equation (\ref{logistic}). We also refer the reader to the paper \cite{Dung} for a small time analysis of $\{L(t)\}_{t\geq 0}$.

\begin{proposition}\label{prop asymp logistic}
Let $\{L(t)\}_{t\geq 0}$ be the unique global strong solution of (\ref{logistic}). Then, 
\begin{itemize}
\item if $a<\sigma^2/2$,
\begin{align}\label{asymp logisitc}
\lim_{t\to+\infty}L(t)=0\quad \mbox{almost surely};
\end{align}
\item if $a\geq\sigma^2/2$, then $L(t)$ is recurrent on $]0,+\infty[$;
\item if $a>\sigma^2/2$, then $L(t)$ converges in distribution, as $t$ tends to infinity, to the unique stationary distribution $\mathtt{Gamma}(\frac{2a}{\sigma^2}-1,\frac{2b}{\sigma^2})$.
\end{itemize}
\end{proposition}

\begin{proof}
See Proposition 3.3 in \cite{GVW}.
\end{proof}

From formula (\ref{explicit logistic}) we see that, for any $t>0$, the random variable $L(t)$ is a function of the Geometric Brownian motion $e^{(a-\sigma^2/2)t+\sigma B(t)}$ and its integral $\int_0^te^{(a-\sigma^2/2)r+\sigma B(r)}dr$. Using the joint probability density function of the random vector
\begin{align*}
\left(e^{(a-\sigma^2/2)t+\sigma B(t)},\int_0^te^{(a-\sigma^2/2)r+\sigma B(r)}dr\right),
\end{align*} 
which can be found in \cite{Yor}, the authors of \cite{Cufaro} write down an expression for the probability density function of $L(t)$: see formula (40) there. However, the authors mention that, due to the presence of oscillating integrals, the numerical treatment of such expression is rather tricky.\\
In the next two results, instead of insisting with exact formulas, we propose upper and lower estimates for the moments $\mathbb{E}[L(t)^p]$ and distribution function $\mathbb{P}(L(t)\leq z)$; the bounds we obtain involve integrals whose numerical approximations do not present the aforementioned difficulties. We also mention the paper \cite{Caravelli} which uses an approach based on power series to approximate the moments of $L(t)$.\\
In the sequel, we will write for $t>0$ 
\begin{align*}
\mathcal{N}_{0,t}(r):=\frac{1}{\sqrt{2t}}e^{-\frac{r^2}{2t}},\quad r\in\mathbb{R}, 
\end{align*}
and
\begin{align*}
\mathcal{N}_{0,t}'(r):=\frac{d}{dr}\mathcal{N}_{0,t}(r)=-\frac{r}{t}\frac{1}{\sqrt{2t}}e^{-\frac{r^2}{2t}},\quad r\in\mathbb{R}.
\end{align*}
For notational convenience we also set
\begin{align}\label{max min}
m(t):=\inf_{r\in [0,t]}B(r)\quad\mbox{ and }\quad M(t):=\sup_{r\in [0,t]}B(r).
\end{align}

\begin{proposition}\label{moments logistic}
	Let $\{L(t)\}_{t\geq 0}$ be the unique global strong solution of (\ref{logistic}). Then, for any $p\geq 0$, we have
	\begin{align}\label{logistic upper moment}
	\mathbb{E}[L(t)^p]\leq 2k_p(t)\int_0^{+\infty}\left(1+b\lambda e^{-\sigma z}K_p(t)\right)^{-p}\mathcal{N}_{0,t}(z)dz,
	\end{align}
	and
	\begin{align}\label{logistic lower moment}
	\mathbb{E}[L(t)^p]\geq 2k_p(t)\int_0^{+\infty}\left(1+b\lambda e^{\sigma z}K_p(t)\right)^{-p}\mathcal{N}_{0,t}(z)dz,
	\end{align}
	where 
	\begin{align*}
	k_p(t):=\lambda^p e^{p(a-\sigma^2/2)t+p^2\sigma^2t/2}\quad\mbox{ and }\quad K_p(t):=\lambda\frac{e^{(a-\sigma^2/2+p\sigma^2)t}-1}{a-\sigma^2/2+p\sigma^2}.
	\end{align*}
\end{proposition}

\begin{proof}
	Fix $p\geq 0$; then,
	\begin{align*}
	\mathbb{E}[L(t)^p]&=\mathbb{E}\left[\frac{\lambda^p e^{p(a-\sigma^2/2)t+p\sigma B(t)}}{\left(1+b\int_0^t\lambda e^{(a-\sigma^2/2)r+\sigma B(r)}dr\right)^p}\right]\\
	&=\mathbb{E}\left[\frac{\lambda^p e^{p(a-\sigma^2/2)t+p^2\sigma^2t/2}e^{p\sigma B(t)-p^2\sigma^2t/2}}{\left(1+b\int_0^t\lambda e^{(a-\sigma^2/2)r+\sigma B(r)}dr\right)^p}\right]\\
	&=\lambda^p e^{p(a-\sigma^2/2)t+p^2\sigma^2t/2}\mathbb{E}\left[\frac{e^{p\sigma B(t)-p^2\sigma^2t/2}}{\left(1+b\int_0^t\lambda e^{(a-\sigma^2/2)r+\sigma B(r)}dr\right)^p}\right]\\
	&=k_p(t)\mathbb{E}\left[\frac{e^{p\sigma B(t)-p^2\sigma^2t/2}}{\left(1+b\int_0^t\lambda e^{(a-\sigma^2/2)r+\sigma B(r)}dr\right)^p}\right].
	\end{align*}
	We now observe that, according to the Girsanov's theorem, for any $T>0$ the law of $\{B(t)\}_{t\in [0,T]}$ under the equivalent probability measure 
	\begin{align*}
	d\mathbb{Q}:=e^{p\sigma B(t)-p^2\sigma^2t/2}d\mathbb{P}\quad\mbox{on }\mathcal{F}_T^{B}
	\end{align*}
	coincides with the one of $\{B(t)+p\sigma t\}_{t\in [0,T]}$ under the measure $\mathbb{P}$. Therefore,
	\begin{align*}
	\mathbb{E}[L(t)^p]&=k_p(t)\mathbb{E}\left[\frac{e^{p\sigma B(t)-p^2\sigma^2t/2}}{\left(1+b\int_0^t\lambda e^{(a-\sigma^2/2)r+\sigma B(r)}dr\right)^p}\right]\\
	&=k_p(t)\mathbb{E}\left[\frac{1}{\left(1+b\int_0^t\lambda e^{(a-\sigma^2/2)r+\sigma (B(r)+p\sigma r)}dr\right)^p}\right]\\
	&=k_p(t)\mathbb{E}\left[\left(1+b\int_0^t\lambda e^{(a-\sigma^2/2)r+\sigma (B(r)+p\sigma r)}dr\right)^{-p}\right].
	\end{align*}
	Now, adopting the notation (\ref{max min}), we can estimate as
	\begin{align*}
	\mathbb{E}[L(t)^p]&=k_p(t)\mathbb{E}\left[\left(1+b\int_0^t\lambda e^{(a-\sigma^2/2)r+\sigma (B(r)+p\sigma r)}dr\right)^{-p}\right]\\
	&\geq k_p(t)\mathbb{E}\left[\left(1+be^{\sigma M(t)}\int_0^t\lambda e^{(a-\sigma^2/2+p\sigma^2)r}dr\right)^{-p}\right]\\
	&= k_p(t)\mathbb{E}\left[\left(1+be^{\sigma M(t)}K_p(t)\right)^{-p}\right],
	\end{align*}
	and similarly
	\begin{align*}
	\mathbb{E}[L(t)^p]&=k_p(t)\mathbb{E}\left[\left(1+b\int_0^t\lambda e^{(a-\sigma^2/2)r+\sigma (B(r)+p\sigma r)}dr\right)^{-p}\right]\\
	&\leq k_p(t)\mathbb{E}\left[\left(1+be^{\sigma m(t)}\int_0^t\lambda e^{(a-\sigma^2/2+p\sigma^2)r}dr\right)^{-p}\right]\\
	&= k_p(t)\mathbb{E}\left[\left(1+be^{\sigma m(t)}K_p(t)\right)^{-p}\right].
	\end{align*}
	Moreover, recalling that, for $A\in\mathcal{B}(\mathbb{R})$ and $t>0$, we have
	\begin{align*}
	\mathbb{P}(m(t)\in A)=2\int_A \mathcal{N}_{0,t}(z)\boldsymbol{1}_{]-\infty,0]}(z)dz\quad\mbox{ and }\quad\mathbb{P}(M(t)\in A)=2\int_A \mathcal{N}_{0,t}(z)\boldsymbol{1}_{[0,+\infty[}(z)dz,
	\end{align*}
	(see formula (8.2) in Chapter 2 from \cite{KS}) we can conclude that
	\begin{align*}
	\mathbb{E}[L(t)^p]&\geq k_p(t)\mathbb{E}\left[\left(1+be^{\sigma M(t)}K_p(t)\right)^{-p}\right]\\
	&=2k_p(t)\int_0^{+\infty}\left(1+be^{\sigma z}K_p(t)\right)^{-p}\mathcal{N}_{0,t}(z)dz,
	\end{align*}
	and
	\begin{align*}
	\mathbb{E}[L(t)^p]&\leq k_p(t)\mathbb{E}\left[\left(1+be^{\sigma m(t)}K_p(t)\right)^{-p}\right]\\
	&=2k_p(t)\int_0^{+\infty}\left(1+be^{-\sigma z}K_p(t)\right)^{-p}\mathcal{N}_{0,t}(z)dz.
	\end{align*}
\end{proof} 

\begin{proposition}\label{df logistic}
	Let $\{L(t)\}_{t\geq 0}$ be the unique global strong solution of (\ref{logistic}). Then, for any $z>0$ and $t>0$, we have the bounds
	\begin{align}\label{upper df logistic}
	\mathbb{P}(L(t)\leq z)\leq -2\int_{\left\{\frac{k(t)e^{\sigma u}}{1+bK(t)e^{\sigma v}}\leq z\right\}\cap\{v>0\}\cap\{u<v\}}\mathcal{N}_{0,t}'(2v-u)dudv,
	\end{align}
	and
	\begin{align}\label{lower df logistic}
	\mathbb{P}(L(t)\leq z)\geq -2\int_{\left\{\frac{k(t)e^{\sigma u}}{1+bK(t)e^{\sigma v}}\leq z\right\}\cap\{v<0\}\cap\{u>v\}}\mathcal{N}_{0,t}'(u-2v)dudv, 
	\end{align}
	with
	\begin{align*}
	k(t):=\lambda e^{(a-\sigma^2/2)t}\quad\mbox{ and }\quad  K(t):=\lambda\frac{e^{(a-\sigma^2/2)t}-1}{a-\sigma^2/2}.
	\end{align*}	
\end{proposition}

\begin{proof}
	We first prove (\ref{lower df logistic}): from (\ref{explicit logistic}) we have 
	\begin{align*}
	L(t)\geq\frac{\lambda e^{(a-\sigma^2/2)t+\sigma B(t)}}{1+be^{\sigma M(t)}\int_0^t\lambda e^{(a-\sigma^2/2)r}dr}=\frac{k(t)e^{\sigma B(t)}}{1+bK(t)e^{\sigma M(t)}}.
	\end{align*}
	The last member above is a function of the two dimensional random vector $(B(t),M(t))$, whose joint probability density function is given by the expression
	\begin{align*}
	f_{B(t),M(t)}(u,v)=
	\begin{cases}
	-2\mathcal{N}_{0,t}'(2v-u),&\mbox{ if $v>0$ and $u<v$},\\
	0,&\mbox{ otherwise}
	\end{cases}
	\end{align*}
	(see formula (8.2) in Chapter 2 from \cite{KS})
	Therefore, for any $z>0$, we obtain
	\begin{align*}
	&\mathbb{P}(L(t)\leq z)\leq\mathbb{P}\left(\frac{k(t)e^{\sigma B(t)}}{1+bK(t)e^{\sigma M(t)}}\leq z\right)\nonumber\\
	&\quad=-2\int_{\left\{\frac{k(t)e^{\sigma u}}{1+bK(t)e^{\sigma v}}\leq z\right\}\cap\{v>0\}\cap\{u<v\}}\mathcal{N}_{0,t}'(2v-u)dudv,
	\end{align*} 
	completing the proof of (\ref{lower df logistic}). Similarly,
	\begin{align*}
	L(t)&\leq\frac{\lambda e^{(a-\sigma^2/2)t+\sigma B(t)}}{1+be^{\sigma m(t)}\int_0^t\lambda e^{(a-\sigma^2/2)r}dr}=\frac{k(t)e^{\sigma B(t)}}{1+bK(t)e^{\sigma m(t)}}.
	\end{align*}
	The last member above is a function of the two dimensional random vector $(B(t),m(t))$, whose joint probability density function is given by the expression
	\begin{align*}
	f_{B(t),m(t)}(u,v)=
	\begin{cases}
	-2\mathcal{N}_{0,t}'(u-2v),&\mbox{ if $v<0$ and $u>v$},\\
	0,&\mbox{ otherwise}.
	\end{cases}
	\end{align*}
	Therefore, for any $z>0$, we obtain
	\begin{align*}
	&\mathbb{P}(L(t)\leq z)\geq\mathbb{P}\left(\frac{k(t)e^{\sigma B(t)}}{1+bK(t)e^{\sigma m(t)}}\leq z\right)\nonumber\\
	&\quad=-2\int_{\left\{\frac{k(t)e^{\sigma u}}{1+bK(t)e^{\sigma v}}\leq z\right\}\cap\{v<0\}\cap\{u>v\}}\mathcal{N}_{0,t}'(u-2v)dudv.
	\end{align*} 
	The proof is complete.
\end{proof}

\begin{remark}
	We observe that the inequality $u<v$ implies
	\begin{align*}
	\frac{k(t)e^{\sigma u}}{1+bK(t)e^{\sigma v}}\leq\frac{k(t)e^{\sigma v}}{1+bK(t)e^{\sigma v}}\leq\frac{k(t)}{bK(t)}.
	\end{align*}
	Therefore, the upper bound (\ref{upper df logistic}) becomes trivial for $z\geq\frac{k(t)}{bK(t)}$; in fact, in that case
	\begin{align*}
	\{u<v\}\Rightarrow\left\{\frac{k(t)e^{\sigma u}}{1+bK(t)e^{\sigma v}}\leq\frac{k(t)}{bK(t)}\right\}\Rightarrow\left\{\frac{k(t)e^{\sigma u}}{1+bK(t)e^{\sigma v}}\leq z\right\} 
	\end{align*}
	which yields
	\begin{align*}
	&\int_{\left\{\frac{k(t)e^{\sigma u}}{1+bK(t)e^{\sigma v}}\leq z\right\}\cap\{u>0\}\cap\{u<v\}}-2\mathcal{N}_{0,t}'(2v-u)dudv\\
	&\quad=\int_{\{v>0\}\cap\{u<v\}}-2\mathcal{N}_{0,t}'(2v-u)dudv=1.
	\end{align*}
\end{remark}

\section{First main theorem: almost sure bounds}

Our first main theorem provides explicit almost sure upper and lower bounds for the solution of (\ref{Mao system}) at any given time $t$. It is useful to introduce the following notation: let 
\begin{align}\label{def L_1}
L_1(t):=\frac{G_1(t)}{1+b_1\int_0^tG_1(r)dr},\quad t\geq 0,
\end{align}
and
\begin{align}\label{def L_2}
L_2(t):=\frac{G_2(t)}{1+b_2\int_0^tG_2(r)dr},\quad t\geq 0,
\end{align}
where for $t\geq 0$ we set
\begin{align*}
G_1(t):=xe^{(a_1-\sigma_1^2/2)t+\sigma_1B_1(t)}\quad\mbox{ and }\quad G_2(t):=ye^{-(a_2+\sigma_2^2/2)t+\sigma_2B_2(t)};
\end{align*}
the parameters $a_1,a_2,b_1,b_2,\sigma_1,\sigma_2,x,y$ are those appearing in equation (\ref{Mao system}). According to the previous section, the stochastic processes $\{L_1(t)\}_{t\geq 0}$ and  $\{L_2(t)\}_{t\geq 0}$ satisfy the equations 
\begin{align}\label{logistic 1}
dL_1(t)=L_1(t)(a_1-b_1L_1(t))dt+\sigma_1L_1(t)dB_1(t),\quad L_1(0)=x,
\end{align}
and
\begin{align}\label{logistic 2}
dL_2(t)=L_2(t)(-a_2-b_2L_2(t))dt+\sigma_2L_2(t)dB_2(t),\quad L_2(0)=y,
\end{align}  
respectively. Therefore, the two dimensional process $\{(L_1(t),L_2(t))\}_{t\geq 0}$ is the unique strong solution of system (\ref{Mao system}) when $c_1=c_2=0$, i.e. when the interaction term  $\frac{X(t)Y(t)}{\beta+Y(t)}$ is not present. 

\subsection{Comments on the use of Girsanov theorem}\label{subsection Girsanov}

We have just mentioned that, by removing the ratio $\frac{X(t)Y(t)}{\beta+Y(t)}$ from its drift, equation (\ref{Mao system}) reduces to the uncoupled system 
	\begin{align}\label{uncoupled}
	\begin{cases}
	dL_1(t)=L_1(t)(a_1-b_1L_1(t))dt+\sigma_1L_1(t)dB_1(t),& L_1(0)=x;\\
	dL_2(t)=L_2(t)(-a_2-b_2L_2(t))dt+\sigma_2L_2(t)dB_2(t),& L_2(0)=y,
	\end{cases}
	\end{align}
whose solution is explicitly represented via formulas (\ref{def L_1}) and (\ref{def L_2}). Since drift removals can in general be performed with the use of Girsanov theorem, one may wonder whether the almost sure properties of (\ref{Mao system}) can be deduced from those of (\ref{uncoupled}) under a suitable equivalent probability measure. Aim of the present subsection is to show that this not case: we are in fact going to prove that the Novikov condition corresponding to the just mentioned drift removal is not fulfilled.

First of all, we notice that system (\ref{uncoupled}) can be rewritten as
	\begin{align*}
	\begin{cases}
	dL_1(t)=L_1(t)(a_1-b_1L_1(t))dt+\sigma_1L_1(t)\left(dB_1(t)+\frac{c_1L_2(t)}{\sigma_1(\beta+L_2(t))}dt-\frac{c_1L_2(t)}{\sigma_1(\beta+L_2(t))}dt\right);\\
	L_1(0)=x;\\
	dL_2(t)=L_2(t)(-a_2-b_2L_2(t))dt+\sigma_2L_2(t)\left(dB_2(t)-\frac{c_2L_1(t)}{\sigma_2(\beta+L_2(t))}dt+\frac{c_2L_1(t)}{\sigma_2(\beta+L_2(t))}dt\right);\\
	L_2(0)=y,
	\end{cases}
	\end{align*}
	or equivalently
	\begin{align}\label{weak solution}
	\begin{cases}
	dL_1(t)=L_1(t)(a_1-b_1L_1(t))dt-c_1\frac{L_1(t)L_2(t)}{\beta+L_2(t)}dt+\sigma_1L_1(t)d\tilde{B_1}(t),& L_1(0)=x;\\
	dL_2(t)=L_2(t)(-a_2-b_2L_2(t))dt+c_2\frac{L_1(t)L_2(t)}{\beta+L_2(t)}dt+\sigma_2L_2(t)d\tilde{B_2}(t), & L_2(0)=y,
	\end{cases}
	\end{align}
	where we set
	\begin{align*}
	\tilde{B_1}(t):=B_1(t)+\int_0^t\frac{c_1L_2(r)}{\sigma_1(\beta+L_2(r))}dr,\quad t\geq 0,
	\end{align*}
	and
	\begin{align*}
	\tilde{B_2}(t):=B_2(t)-\int_0^t\frac{c_2L_1(r)}{\sigma_2(\beta+L_2(r))}dr,\quad t\geq 0.
	\end{align*}
	Now, if the Novikov condition
	\begin{align}\label{Novikov}
	\mathbb{E}\left[\exp\left\{\frac{1}{2}\int_0^T\left(\frac{c_1L_2(r)}{\sigma_1(\beta+L_2(r))}\right)^2+\left(\frac{c_2L_1(r)}{\sigma_2(\beta+L_2(r))}\right)^2dr\right\}\right]<+\infty
	\end{align}
	is satisfied for some $T>0$, then the stochastic process $\{(\tilde{B_1}(t),\tilde{B_1}(t))\}_{t\in [0,T]}$ is according to the Girsanov theorem a standard two dimensional Brownian motion on the probability space $(\Omega,\mathcal{F}_T,\mathbb{Q})$ (here $\{\mathcal{F}_t\}_{t\geq 0}$ denotes the augmented Brownian filtration) with
	\begin{align*}
	&d\mathbb{Q}:=\exp\left\{-\int_0^T\frac{c_1L_2(r)}{\sigma_1(\beta+L_2(r))}dB_1(r)-\frac{1}{2}\int_0^T\left(\frac{c_1L_2(r)}{\sigma_1(\beta+L_2(r))}\right)^2dr\right\}\\
	&\quad\times \exp\left\{\int_0^T\frac{c_2L_1(r)}{\sigma_2(\beta+L_2(r))}dB_2(r)-\frac{1}{2}\int_0^T\left(\frac{c_2L_1(r)}{\sigma_2(\beta+L_2(r))}\right)^2dr\right\}d\mathbb{P}.
	\end{align*}
	Moreover, in this case equation (\ref{weak solution}) implies that the two dimensional process $\{(L_1(t),L_2(t))\}_{t\in [0,T]}$ is a weak solution of (\ref{Mao system}) with respect to $(\Omega,\{\mathcal{F}_t\}_{t\in [0,T]},\mathbb{Q},\{(\tilde{B_1}(t),\tilde{B_1}(t))\}_{t\in [0,T]})$.\\
	We now prove that condition (\ref{Novikov}) cannot be true without additional assumptions on the parameters of our model. In fact, 
	\begin{align*}
	&\mathbb{E}\left[\exp\left\{\frac{1}{2}\int_0^T\left(\frac{c_1L_2(r)}{\sigma_1(\beta+L_2(r))}\right)^2+\left(\frac{c_2L_1(r)}{\sigma_2(\beta+L_2(r))}\right)^2dr\right\}\right]\\
	&\quad\geq \mathbb{E}\left[\exp\left\{\frac{1}{2}\int_0^T\left(\frac{c_2L_1(r)}{\sigma_2(\beta+L_2(r))}\right)^2dr\right\}\right]\\
	&\quad=\mathbb{E}\left[\exp\left\{\frac{c_2^2}{2\sigma_2^2}\int_0^T\frac{L^2_1(r)}{(\beta+L_2(r))^2}dr\right\}\right]\\
	&\quad\geq\mathbb{E}\left[\exp\left\{\frac{c_2^2}{2\sigma_2^2 \mathcal{M}_2}\int_0^TL^2_1(r)dr\right\}\right]\\
	\end{align*}
	where we introduced the notation
	\begin{align*}
	\mathcal{M}_2:=\sup_{r\in [0,T]}(\beta+L_2(r))^2.
	\end{align*}
	We now apply Jensen's inequality to the Lebesgue integral and use the identity
	\begin{align*}
	\int_0^TL_1(r)dr=\frac{1}{b_1}\ln\left(1+b_1\int_0^TG_1(r)dr\right)
	\end{align*}
	to get
	\begin{align*}
	&\mathbb{E}\left[\exp\left\{\frac{1}{2}\int_0^T\left(\frac{c_1L_2(r)}{\sigma_1(\beta+L_2(r))}\right)^2+\left(\frac{c_2L_1(r)}{\sigma_2(\beta+L_2(r))}\right)^2dr\right\}\right]\\
	&\quad\geq\mathbb{E}\left[\exp\left\{\frac{c_2^2}{2\sigma_2^2 \mathcal{M}_2}\int_0^TL^2_1(r)dr\right\}\right]\\
	&\quad=\mathbb{E}\left[\exp\left\{\frac{c_2^2T}{2\sigma_2^2 \mathcal{M}_2T}\int_0^TL^2_1(r)dr\right\}\right]\\
	&\quad\geq\mathbb{E}\left[\exp\left\{\frac{c_2^2}{2\sigma_2^2 \mathcal{M}_2T}\left(\int_0^TL_1(r)dr\right)^2\right\}\right]\\
	&\quad=\mathbb{E}\left[\exp\left\{\frac{c_2^2}{2\sigma_2^2 \mathcal{M}_2Tb_1^2}\left(\ln\left(1+b_1\int_0^TG_1(r)dr\right)\right)^2\right\}\right]\\
	&\quad\geq \mathbb{E}\left[\exp\left\{\frac{c_2^2}{2\sigma_2^2 \mathcal{M}_2Tb_1^2}\left(\ln\left(1+b_1K_1(T)e^{\sigma_1m_1(T)}\right)\right)^2\right\}\right]\\
	&\quad\geq \mathbb{E}\left[\exp\left\{\frac{c_2^2}{2\sigma_2^2 \mathcal{M}_2Tb_1^2}\left(\ln\left(b_1K_1(T)e^{\sigma_1m_1(T)}\right)\right)^2\right\}\right]\\
	&\quad=\mathbb{E}\left[\exp\left\{\frac{c_2^2}{2\sigma_2^2 \mathcal{M}_2Tb_1^2}\left(\sigma_1m_1(T)+\ln(b_1K_1(T))\right)^2\right\}\right].
	\end{align*}
	Here, we set 
	\begin{align*}
	K_1(T)=\frac{e^{(a_1-\sigma_1^2/2)T}-1}{a_1-\sigma_1^2/2}\quad\mbox{ and }\quad m_1(T):=\min_{t\in [0,T]}B_1(t).
	\end{align*}  
Using the independence between $B_1$ and $B_2$, we can write the last expectation as 
\begin{align*}
&\mathbb{E}\left[\exp\left\{\frac{c_2^2}{2\sigma_2^2 \mathcal{M}_2Tb_1^2}\left(\sigma_1m_1(T)+\ln(b_1K_1(T))\right)^2\right\}\right]\\
&\quad=\int_{\beta^2}^{+\infty}\left(\int_{-\infty}^{0}e^{\frac{C}{2Tz}\left(\sigma_1u+D\right)^2}\frac{2}{\sqrt{2\pi T}}e^{-\frac{u^2}{2T}}du\right)d\mu(z),
\end{align*}
where $\mu$ stands for the law of $\mathcal{M}_2$, $C:=\frac{c_2^2}{\sigma_2^2 b_1^2}$ and $D:=\ln(b_1K_1(T))$. It is now clear that the inner integral above is finite if and only if $z\geq C\sigma_1^2$. Since $z$ ranges in the interval $]\beta^2 ,\infty[$, we deduce that the last condition is verified for all $z\in ]\beta^2,+\infty[$ only when $\beta^2\geq C\sigma_1^2$, which in our notation means
\begin{align}\label{non go girsanov}
\beta\geq\frac{c_2\sigma_1}{b_1\sigma_2}.
\end{align} 
Therefore, if the parameters describing system (\ref{Mao system}) do not respect the bound (\ref{non go girsanov}), then inequality
\begin{align*}
&\mathbb{E}\left[\exp\left\{\frac{1}{2}\int_0^T\left(\frac{c_1L_2(r)}{\sigma_1(\beta+L_2(r))}\right)^2+\left(\frac{c_2L_1(r)}{\sigma_2(\beta+L_2(r))}\right)^2dr\right\}\right]\\
&\quad\geq 2\int_{\beta^2}^{+\infty}\left(\int_{-\infty}^{0}e^{\frac{C}{2Tz}\left(\sigma_1u+D\right)^2}\frac{1}{\sqrt{2\pi T}}e^{-\frac{u^2}{2T}}du\right)d\mu(z)=+\infty,
\end{align*}
which is valid for all $T>0$, implies the failure of Novikov condition (\ref{Novikov}). From this point of view the almost sure properties of the solution of (\ref{Mao system}) cannot be deduced from those of the uncoupled system (\ref{uncoupled}).

\begin{remark}
The functional response in the foraging arena model formally appears to be a particular case of the one that characterizes the Beddington-DeAngelis model (take $\alpha_1=0$). However, referring to the change of measure technique mentioned above, we see that the Novikov condition corresponding to the Beddington-DeAngelis model would amount at the finiteness of    
\begin{align*}
\mathbb{E}\left[\exp\left\{\frac{1}{2}\int_0^T\left(\frac{c_1L_2(r)}{\sigma_1(\beta+\alpha_1 L_1(r)+\alpha_2L_2(r))}\right)^2+\left(\frac{c_2L_1(r)}{\sigma_2(\beta+\alpha_1 L_1(r)+\alpha_2L_2(r))}\right)^2dr\right\}\right].
\end{align*}
Since the two ratios in the Lebesgue integral are upper bounded almost surely by $\frac{c_1}{\sigma_1\alpha_2}$ and $\frac{c_2}{\sigma_2\alpha_1}$, respectively, we get immediately the finiteness, for all $T>0$, of the expectation above. Therefore, in the Beddington-DeAngelis model one may utilize the change of measure approach to study almost sure properties of the solution on any finite interval of time $[0,T]$. The same reasoning applies also to the Crowley-Martin and Hassell-Varley functional responses.
\end{remark}	
	
\subsection{Statement and proof of the first main theorem}

Recall that, according to the discussion in Section 1, the quantity 
\begin{align*}
\phi:=\frac{\sigma_1^2}{2}+\frac{b_1\beta a_2}{c_2}+\frac{b_1\beta \sigma_2^2}{2c_2}
\end{align*} 
is a threshold determining the asymptotic behaviour of $X(t)$ and $Y(t)$.

\begin{theorem}\label{a.s. bounds}
	Let $\{(X(t),Y(t))\}_{t\geq 0}$ be the unique global strong solution of (\ref{Mao system}). Then, for all $t\geq 0$ the following bounds hold almost surely:
	\begin{align}\label{boundY}
	L_2(t)\leq Y(t)\leq L_2(t)\left(1+b_1\int_0^t G_1(r)dr\right)^{\frac{c_2}{\beta b_1}};
	\end{align}
	if $a_1<\phi$, then
	\begin{align}\label{boundX}
	L_1(t)e^{-\frac{c_1}{\beta b_2}\left(1+b_1\int_0^t G_1(r)dr\right)^{\frac{c_2}{\beta b_1}}\ln\left(1+b_2\int_0^t G_2(r)dr\right)}\leq X(t)\leq L_1(t);
	\end{align}
	if $a_1>\phi$, then
	\begin{align}\label{boundX2}
	L_1(t)e^{-c_1t}\leq X(t)\leq L_1(t).
	\end{align}	
\end{theorem}

\begin{remark}
We assumed at the beginning of this manuscript that the Brownian motions $\{B_1(t)\}_{t\geq 0}$ and $\{B_2(t)\}_{t\geq 0}$, driving the two dimensional system (\ref{Mao system}), are independent. However, this assumption is not needed in the derivation of the almost sure bounds stated above, as long as system (\ref{Mao system}) possesses a positive global strong solution. Therefore, the estimates (\ref{boundY}), (\ref{boundX}) and (\ref{boundX2}) remain true in the case of correlated Brownian motions as well.  
\end{remark}

\begin{remark}\label{comparison with asymptotic}
The bounds in Theorem \ref{a.s. bounds} are consistent with the asymptotic results obtained in \cite{Mao foraging}. In fact: 
\begin{itemize}
\item $a_1<\frac{\sigma_1^2}{2}$: taking the limit as $t$ tends to infinity in the second inequality of (\ref{boundX}) we get
\begin{align*}
0\leq \lim_{t\to +\infty}X(t)\leq \lim_{t\to +\infty}L_1(t),
\end{align*}
which, in combination with (\ref{asymp logisitc}) for $L_1$, gives
\begin{align*}
\lim_{t\to +\infty}X(t)=0.
\end{align*}
On the other hand, if we take the limit in (\ref{boundY}) we obtain  
\begin{align*}
0\leq\lim_{t\to +\infty}Y(t)\leq \lim_{t\to +\infty}L_2(t)\left(1+b_1\int_0^t G_1(r)dr\right)^{\frac{c_2}{\beta b_1}}.
\end{align*}
According to formula 1.8.4 page 612 in \cite{BS} the random variable $\int_0^{+\infty}G_1(r)dr$ is finite almost surely; this fact and (\ref{asymp logisitc}) for $L_2$ yield
\begin{align*}
\lim_{t\to +\infty}Y(t)=0,
\end{align*}
completing the proof of (\ref{Mao asymp 1});
\item $\frac{\sigma_1^2}{2}<a_1<\phi=\frac{\sigma_1^2}{2}+\frac{b_1\beta a_2}{c_2}+\frac{b_1\beta \sigma_2^2}{2c_2}$: first of all, we write 
\begin{align*}
L_2(t)\leq G_2(t)=e^{-(a_2+\sigma_2^2/2)t+\sigma_2B_2(t)};
\end{align*}
moreover, since 
\begin{align*}
\int_0^tG_1(r)ds\leq e^{\sigma_1 M_1(t)}K_1(t), 
\end{align*}
where $M_1(t):=\max_{t\in [0,t]}B_1(r)$ and
\begin{align*}
K_1(t):=x\frac{e^{(a_1-\sigma_1^2/2)t}-1}{a_1-\sigma_1^2/2},
\end{align*}
we get
\begin{align*}
\left(1+b_1\int_0^t G_1(r)dr\right)^{\frac{c_2}{\beta b_1}}&\leq \left(1+b_1e^{\sigma_1 M_1(t)}K_1(t)\right)^{\frac{c_2}{\beta b_1}}\\
&\leq \left(1+Ce^{\sigma_1 M_1(t)}e^{(a_1-\sigma_1^2/2)t}\right)^{\frac{c_2}{\beta b_1}},
\end{align*}
for a suitable positive constant $C$. Therefore, 
\begin{align}\label{par}
&L_2(t)\left(1+b_1\int_0^t G_1(r)dr\right)^{\frac{c_2}{\beta b_1}}\leq e^{-(a_2+\sigma_2^2/2)t+\sigma_2B_2(t)}\left(1+Ce^{\sigma_1 M_1(t)}e^{(a_1-\sigma_1^2/2)t}\right)^{\frac{c_2}{\beta b_1}}\nonumber\\
&\quad=\left(e^{-\frac{(a_2+\sigma_2^2/2)\beta b_1}{c_2}t+\frac{\sigma_2\beta b_1}{c_2}B_2(t)}+Ce^{-\frac{(a_2+\sigma_2^2/2)\beta b_1}{c_2}t+\frac{\sigma_2\beta b_1}{c_2}B_2(t)}e^{\sigma_1 M_1(t)}e^{(a_1-\sigma_1^2/2)t}\right)^{\frac{c_2}{\beta b_1}}\nonumber\\
&\quad=\left(e^{-\frac{(a_2+\sigma_2^2/2)\beta b_1}{c_2}t+\frac{\sigma_2\beta b_1}{c_2}B_2(t)}+Ce^{\left(a_1-\sigma_1^2/2-\frac{(a_2+\sigma_2^2/2)\beta b_1}{c_2}\right)t+\frac{\sigma_2\beta b_1}{c_2}B_2(t)}e^{\sigma_1 M_1(t)}\right)^{\frac{c_2}{\beta b_1}}.
\end{align}
Recalling that
\begin{align*}
\mathbb{P}\left(\lim_{t\to+\infty}\frac{B(t)}{t}=0\right)=\mathbb{P}\left(\lim_{t\to+\infty}\frac{M_1(t)}{t}=0\right)=1,
\end{align*}
(see for instance \cite{Mao book}), we can say that both terms inside the parenthesis in (\ref{par}) will tend to zero as $t$ tends to infinity if the constants multiplying $t$ in the exponentials are negative. While this is obvious for the first exponential, the negativity of the constant 
\begin{align*}
a_1-\sigma_1^2/2-\frac{(a_2+\sigma_2^2/2)\beta b_1}{c_2}
\end{align*}
is equivalent to the condition $a_1<\phi$, i.e. the regime under consideration. Hence, passing to the limit in (\ref{boundY}), we conclude that
\begin{align*}
\lim_{t\to+\infty}Y(t)=0;
\end{align*}
this corresponds to (\ref{Mao asymp 2}). In addition, from (\ref{boundX}) we obtain 
\begin{align*}
\lim_{t\to +\infty}\frac{1}{t}\int_0^tX(r)dr\leq\lim_{t\to +\infty}\frac{1}{t}\int_0^tL_1(r)dr=\frac{a_1-\sigma_1^2/2}{b_1}.
\end{align*}
Here, we utilized Proposition \ref{prop asymp logistic} for $L_1$ with $a_1>\sigma_1^2/2$, in particular the ergodic property
\begin{align*}
\lim_{t\to +\infty}\frac{1}{t}\int_0^tL_1(r)dr=\mathbb{E}[L_{\infty}],
\end{align*}
with $\mathbb{E}[L_{\infty}]$ being the expectation of the unique stationary distribution. This partially proves (\ref{Mao asymp 2 bis}). 
\end{itemize}
\end{remark}

\begin{proof}
We start finding the It\^o's differential of the stochastic process $\frac{1}{L_1(t)}$: 
\begin{align*}
d\frac{1}{L_1(t)}&=-\frac{1}{L_1^2(t)}dL_1(t)+\frac{1}{L_1^3(t)}d\langle L_1\rangle_t\\
&=-\frac{a_1-b_1L_1(t)}{L_1(t)}dt-\frac{\sigma_1}{L_1(t)}dB_1(t)+\frac{\sigma_1^2}{L_1(t)}dt\\
&=\frac{\sigma_1^2-a_1+b_1L_1(t)}{L_1(t)}dt-\frac{\sigma_1}{L_1(t)}dB_1(t).
\end{align*}
Combining this expression with the first equation in (\ref{Mao system}) we get
\begin{align*}
d \frac{X(t)}{L_1(t)}=& X(t)d\frac{1}{L_1(t)}+\frac{1}{L_1(t)}dX(t)+d\left\langle X, 1/L_1\right\rangle (t)\\
=& X(t)\left(\frac{\sigma_1^2-a_1+b_1L_1(t)}{L_1(t)}dt-\frac{\sigma_1}{L_1(t)}dB_1(t)\right)\\
&+\frac{1}{L_1(t)}\left[X(t)\left(a_1-b_1X(t)-\frac{c_1Y(t)}{\beta+Y(t)}\right)dt+\sigma_1 X(t)dB_1(t)\right]\\
&-\sigma_1^2\frac{X(t)}{L_1(t)}dt\\
=&\frac{X(t)}{L_1(t)}\left[\sigma_1^2-a_1+b_1L_1(t)+a_1-b_1X(t)-\frac{c_1Y(t)}{\beta+Y(t)}-\sigma_1^2\right]dt\\
=&\frac{X(t)}{L_1(t)}\left[b_1(L_1(t)-X(t))-\frac{c_1Y(t)}{\beta+Y(t)}\right]dt.
\end{align*}
Since $\frac{X(0)}{L_1(0)}=1$, the last chain of equalities implies
\begin{align}\label{c}
\frac{X(t)}{L_1(t)}=\exp\left\{b_1\int_0^t(L_1(r)-X(r))dr-c_1\int_0^t\frac{Y(r)}{\beta+Y(r)}dr\right\}.
\end{align}
Following the previous reasoning we also find that 
\begin{align*}
d\frac{1}{L_2(t)}&=-\frac{1}{L_2^2(t)}dL_2(t)+\frac{1}{L_2^3(t)}d\langle L_2\rangle_t\\
&=-\frac{-a_2-b_2L_2(t)}{L_2(t)}dt-\frac{\sigma_2}{L_2(t)}dB_2(t)+\frac{\sigma_2^2}{L_2(t)}dt\\
&=\frac{\sigma_2^2+a_2+b_2L_2(t)}{L_2(t)}dt-\frac{\sigma_2}{L_2(t)}dB_2(t).
\end{align*}
Combining this expression with the second equation in (\ref{Mao system}) we get
\begin{align*}
d \frac{Y(t)}{L_2(t)}=& Y(t)d\frac{1}{L_2(t)}+\frac{1}{L_2(t)}dY(t)+d\left\langle Y, 1/L_2\right\rangle (t)\\
=& Y(t)\left(\frac{\sigma_2^2+a_2+b_2L_2(t)}{L_2(t)}dt-\frac{\sigma_2}{L_2(t)}dB_2(t)\right)\\
&+\frac{1}{L_2(t)}\left[Y(t)\left(-a_2-b_2X(t)+\frac{c_2X(t)}{\beta+Y(t)}\right)dt+\sigma_2 Y(t)dB_2(t)\right]\\
&-\sigma_2^2\frac{Y(t)}{L_2(t)}dt\\
=&\frac{Y(t)}{L_2(t)}\left[\sigma_2^2+a_2+b_2L_2(t)-a_2-b_2Y(t)+\frac{c_2X(t)}{\beta+Y(t)}-\sigma_2^2\right]dt\\
=&\frac{Y(t)}{L_2(t)}\left[b_2(L_2(t)-Y(t))+\frac{c_2X(t)}{\beta+Y(t)}\right]dt.
\end{align*}
Since $\frac{Y(0)}{L_2(0)}=1$, the last chain of equalities implies
\begin{align}\label{d}
\frac{Y(t)}{L_2(t)}=\exp\left\{b_2\int_0^t(L_2(r)-Y(r))dr+c_2\int_0^t\frac{X(r)}{\beta+Y(r)}dr\right\}.
\end{align}
We now observe that 
\begin{align*}
\mathbb{P}\left(\frac{X(t)Y(t)}{\beta+Y(t)}>0\right)=1,\quad\mbox{ for any $t\geq 0$}
\end{align*}
(remember that $X(t)$ and $Y(t)$ are positive for all $t\geq 0$); therefore, by means of standard comparison theorems for SDEs (see for instance Theorem 1.1 in Chapter VI from \cite{Ikeda Watanabe}) applied to (\ref{Mao system}) we deduce that 
\begin{align}\label{e}
X(t)\leq L_1(t),\quad\mbox{for all $t\geq 0$}, 
\end{align}
and 
\begin{align}\label{f}
Y(t)\geq L_2(t),\quad\mbox{for all $t\geq 0$}, 
\end{align}
where $\{L_1(t)\}_{t\geq 0}$ and $\{L_2(t)\}_{t\geq 0}$ solve (\ref{logistic 1}) and (\ref{logistic 2}), respectively. Therefore, equation (\ref{c}) leads to
\begin{align*}
\exp\left\{-c_1\int_0^t\frac{Y(r)}{\beta+Y(r)}dr\right\}\leq \frac{X(t)}{L_1(t)}\leq 1,
\end{align*}
or equivalently,
\begin{align}\label{cc}
L_1(t)\exp\left\{-c_1\int_0^t\frac{Y(r)}{\beta+Y(r)}dr\right\}\leq X(t)\leq L_1(t),
\end{align}
while equation (\ref{d}) leads to
\begin{align*}
1\leq \frac{Y(t)}{L_2(t)}\leq\exp\left\{c_2\int_0^t\frac{X(r)}{\beta+Y(r)}dr\right\},
\end{align*}
or equivalently,
\begin{align}\label{dd}
L_2(t)\leq Y(t)\leq L_2(t)\exp\left\{c_2\int_0^t\frac{X(r)}{\beta+Y(r)}dr\right\}.
\end{align}
The lower bound in (\ref{cc}) and upper bound in (\ref{dd}) are not explicit yet since they depend on the solution itself. To solve this problem we first recall that the process $\{L_2(t)\}_{t\geq 0}$ is positive and converges almost surely to zero exponentially fast, as $t$ tends to infinity. Now, by virtue of (\ref{e}), (\ref{f}) and the infinitesimal behaviour of $L_2$, we can upper bound the right hand side in (\ref{dd}) as
\begin{align*}
L_2(t)\exp\left\{c_2\int_0^t\frac{X(r)}{\beta+Y(r)}dr\right\}&\leq L_2(t)\exp\left\{c_2\int_0^t\frac{L_1(r)}{\beta+L_2(r)}dr\right\}\\
&\leq L_2(t)\exp\left\{\frac{c_2}{\beta}\int_0^tL_1(r)dr\right\},
\end{align*}
In addition, since
\begin{align*}
L_1(t)=\frac{1}{b_1}\frac{d}{dt}\ln\left(1+b_1\int_0^tG_1(r)dr\right),
\end{align*}
the last member above can be rewritten as
\begin{align*}
L_2(t)\exp\left\{\frac{c_2}{\beta}\int_0^tL_1(r)dr\right\}&=L_2(t)\exp\left\{\frac{c_2}{\beta b_1}\ln\left(1+b_1\int_0^tG_1(r)dr\right)\right\}\\
&=L_2(t)\left(1+b_1\int_0^tG_1(r)dr\right)^{\frac{c_2}{\beta b_1}}.
\end{align*}
Combining this estimate with (\ref{dd}) we obtain (\ref{boundY}). 

\begin{figure}[H]
	\centering
	\includegraphics[width=0.7\linewidth]{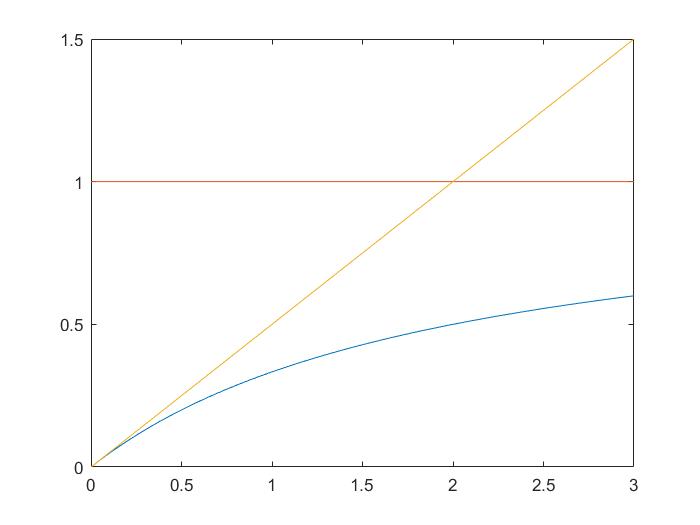}
	\caption{Upper bounds for the function $y\mapsto \frac{y}{2+y}$ (green line) with the function $y\mapsto\frac{y}{2}$ (yellow line) and the function $y\mapsto 1$ (red line).}
\end{figure}

For the lower bound in (\ref{cc}), we observe that the function $y\mapsto\frac{y}{\beta+y}$, for $y>0$, can be sharply upper bounded by affine functions in two different ways: the upper bound $y\mapsto 1$ is sharp at infinity but not accurate at zero while the upper bound $y\mapsto \frac{y}{\beta}$ is sharp at zero but very bad at infinity. Therefore, according to the asymptotic results proved in \cite{Mao foraging} and mentioned in the Introduction, we now proceed distinguishing two different regimes:
\begin{itemize}
	\item when $a_1<\phi$, the process $\{Y_t\}_{t\geq 0}$ tends to zero exponentially fast and hence we utilize the process $\frac{Y_r}{\beta}$ to upper bound $\frac{Y_r}{\beta+Y_r}$. The left hand side of (\ref{cc}) is then simplified to
\begin{align}\label{ccc}
&L_1(t)\exp\left\{-c_1\int_0^t\frac{Y(r)}{\beta+Y(r)}dr\right\}\geq L_1(t)\exp\left\{-\frac{c_1}{\beta}\int_0^tY(r)dr\right\}\nonumber\\
&\quad\geq L_1(t)\exp\left\{-\frac{c_1}{\beta}\int_0^tL_2(r)\left(1+b_1\int_0^rG_1(u)du\right)^{\frac{c_2}{\beta b_1}}dr\right\}\nonumber\\
&\quad\geq L_1(t)\exp\left\{-\frac{c_1}{\beta}\left(1+b_1\int_0^tG_1(r)dr\right)^{\frac{c_2}{\beta b_1}}\int_0^tL_2(r)dr\right\}\nonumber\\
&=L_1(t)\exp\left\{-\frac{c_1}{\beta b_2}\left(1+b_1\int_0^tG_1(r)dr\right)^{\frac{c_2}{\beta b_1}}\ln\left(1+b_2\int_0^tG_2(r)dr\right)\right\}.
\end{align}
Here, in the second inequality we utilized the upper bound in (\ref{boundY}) while in the last equality we employed the identity
\begin{align*}
L_2(t)=\frac{1}{b_2}\frac{d}{dt}\ln\left(1+b_2\int_0^tG_2(r)dr\right).
\end{align*}
Inserting (\ref{ccc}) in the left hand side of (\ref{cc}), one gets (\ref{boundX});
\item when $a_1>\phi$, the process $\{Y_t\}_{t\geq 0}$ has a more oscillatory behaviour; therefore, we prefer to upper bound the ratio $\frac{Y_r}{\beta+Y_r}$ with one. This gives
\begin{align*}
L_1(t)\exp\left\{-c_1\int_0^t\frac{Y(r)}{\beta+Y(r)}dr\right\}\geq L_1(t)e^{-c_1t},
\end{align*}
and (\ref{cc}) reduces to (\ref{boundX2}).
\end{itemize}
\end{proof}

\begin{remark}\label{all parameters}
It is important to emphasize that both the lower bounds in (\ref{boundX}) and (\ref{boundX2}) remain valid without restrictions on the parameters: this is clear from the proof of Theorem \ref{a.s. bounds} and in particular from the use of the comparison principle we made. In fact, one may combine the two lower estimates as
\begin{align*}
L_1(t)\max\left\{e^{-\frac{c_1}{\beta b_2}\left(1+b_1\int_0^t G_1(r)dr\right)^{\frac{c_2}{\beta b_1}}\ln\left(1+b_2\int_0^t G_2(r)dr\right)},e^{-c_1t}\right\}\leq X(t)\leq L_1(t),
\end{align*}
and argue on the different values attained by the maximum above. However, such analysis would necessarily involve the non directly observable quantities $\int_0^t G_1(r)dr$, $\int_0^t G_2(r)dr$ and their probabilities. That is why we preferred to suggest which lower bound is better suited for the given set of parameters.
\end{remark}

\section{Second main theorem: bounds for the moments}

The next theorem presents upper and lower estimates for the joint moments of $X(t)$ and $Y(t)$ at any given time $t$. These bounds, which rely on the almost sure inequalities (\ref{boundY}), (\ref{boundX}) and (\ref{boundX2}) are represented through closed form expressions involving Lebesgue integrals; such integrals can be evaluated via numerical approximations or Monte Carlo simulations. \\
We also mention that in \cite{Mao foraging} the authors prove an asymptotic upper bound for the moments $\mathbb{E}\left[(X(t)^2+Y(t)^2)^{\theta/2}\right]$ with $\theta$ being a positive real number.

\begin{theorem}\label{main theorem moments}
Let $\{(X(t),Y(t))\}_{t\geq 0}$ be the unique global strong solution of (\ref{Mao system}). For all $t\geq 0$ we have the following estimates:
\begin{enumerate}
\item  if $p,q\geq 0$ with $\frac{qc_2}{\beta b_1}-p\geq 1$, then
\begin{align}\label{joint upper moment}
\mathbb{E}\left[X(t)^pY(t)^q\right]\leq& 2k_{1,p}(t)k_{2,q}(t)\left(1+b_1x\frac{e^{\left(a_1+\left(\frac{qc_2}{\beta b_1}+p-1\right)\frac{\sigma_1^2}{2}\right)t}-1}{a_1+\left(\frac{qc_2}{\beta b_1}+p-1\right)\frac{\sigma_1^2}{2}}\right)^{\frac{qc_2}{\beta b_1}-p}\nonumber\\
&\quad\times \int_0^{+\infty}\left(1+b_2y e^{-\sigma_2 z}K_{2,q}(t)\right)^{-q}\mathcal{N}_{0,t}(z)dz. 
\end{align}
\item if $p,q\geq 0$ and $a_1>\phi$, then
\begin{align}\label{joint lower moment}
\mathbb{E}\left[X(t)^pY(t)^q\right]\geq& 4e^{-pc_1t} k_{1,p}(t)k_{2,q}(t)\int_0^{+\infty}\left(1+b_1x e^{\sigma_1 z}K_{1,p}(t)\right)^{-p}\mathcal{N}_{0,t}(z)dz\nonumber\\
&\quad\times\int_0^{+\infty}\left(1+b_2y e^{\sigma_2 z}K_{2,q}(t)\right)^{-q}\mathcal{N}_{0,t}(z)dz.
\end{align}
\item if $p,q\geq 0$ and $a_1<\phi$, then
\begin{align}\label{lower moments X}
&\mathbb{E}[X(t)^p]\geq -4k_1(t)^p\int_A\frac{e^{p\sigma_1u_1-\frac{pc_1}{\beta b_2}\left(1+b_1K_1(t)e^{\sigma_1v_1}\right)^{\frac{c_2}{\beta b_1}}\ln\left(1+b_2K_2(t)e^{\sigma_2v_2}\right)}}{(1+b_1K_1(t)e^{\sigma_1v_1})^p}\nonumber\\
&\quad\quad\quad\quad\quad\quad\times\mathcal{N}_{0,t}'(2v_1-u_1)\mathcal{N}_{0,t}(v_2)du_1dv_1dv_2,
\end{align}
where 
\begin{align*}
A:=\{(u_1,v_1,v_2)\in\mathbb{R}^3: v_1>0,u_1<v_1,v_2>0\},
\end{align*}
while
\begin{align}\label{lower moments Y}
\mathbb{E}[Y(t)^q]\geq 2k_{2,q}(t)\int_0^{+\infty}\left(1+b_2y e^{\sigma_2 z}K_{2,q}(t)\right)^{-q}\mathcal{N}_{0,t}(z)dz.
\end{align}	
\end{enumerate}
Here,
\begin{align*}
&k_1(t):=xe^{(a_1-\sigma_1^2/2)t},\quad K_1(t):=x\frac{e^{(a_1-\sigma_1^2/2)t}-1}{a_1-\sigma_1^2/2},\quad K_2(t):=y\frac{e^{(a_2-\sigma_2^2/2)t}-1}{a_2-\sigma_2^2/2},\\ 
&k_{1,p}(t):=x^p e^{p(a_1-\sigma_1^2/2)t+p^2\sigma_1^2t/2},\quad\quad K_{1,p}(t):=x\frac{e^{(a_1-\sigma_1^2/2+p\sigma_1^2)t}-1}{a_1-\sigma_1^2/2+p\sigma_1^2},\\
&k_{2,p}(t):=y^p e^{p(a_2-\sigma_2^2/2)t+p^2\sigma_2^2t/2}\quad\quad K_{2,p}(t):=y\frac{e^{(a_2-\sigma_2^2/2+p\sigma_2^2)t}-1}{a_2-\sigma_2^2/2+p\sigma_2^2}.
\end{align*}
\end{theorem}

\begin{proof}
\begin{enumerate}	
\item Using (\ref{boundY}) and (\ref{boundX}) (or (\ref{boundX2})), we can write 
\begin{align*}
\mathbb{E}\left[X(t)^pY(t)^q\right]\leq&\mathbb{E}\left[L_1(t)^p L_2(t)^q\left(1+b_1\int_0^tG_1(r)dr\right)^{\frac{qc_2}{\beta b_1}}\right]\\
=&\mathbb{E}\left[L_1(t)^p\left(1+b_1\int_0^tG_1(r)dr\right)^{\frac{qc_2}{\beta b_1}}\right]\mathbb{E}\left[L_2(t)^q\right]\\
=&\mathbb{E}\left[\frac{G_1(t)^p}{\left(1+b_1\int_0^tG_1(r)dr\right)^p}\left(1+b_1\int_0^tG_1(r)dr\right)^{\frac{qc_2}{\beta b_1}}\right]\mathbb{E}\left[L_2(t)^q\right]\\
=&\mathbb{E}\left[G_1(t)^p\left(1+b_1\int_0^tG_1(r)dr\right)^{\frac{qc_2}{\beta b_1}-p}\right]\mathbb{E}\left[L_2(t)^q\right]\\
=&\mathcal{I}_1\mathcal{I}_2,
\end{align*} 
where we set
\begin{align*}
\mathcal{I}_1:=\mathbb{E}\left[G_1(t)^p\left(1+b_1\int_0^tG_1(r)dr\right)^{\frac{qc_2}{\beta b_1}-p}\right]\quad\mbox{ and }\quad\mathcal{I}_2:=\mathbb{E}\left[L_2(t)^q\right].
\end{align*}
From (\ref{logistic upper moment}) we get immediately that
\begin{align*}
\mathcal{I}_2\leq 2k_{2,q}(t)\int_0^{+\infty}\left(1+b_2y e^{-\sigma_2 z}K_{2,q}(t)\right)^{-q}\mathcal{N}_{0,t}(z)dz.
\end{align*}
Now, mimicking the proof of Proposition \ref{moments logistic} we can write
\begin{align*}
\mathcal{I}_1&=k_{1,p}(t)\mathbb{E}\left[e^{p\sigma_1B_1(t)-p^2\sigma_1^2t/2}\left(1+b_1\int_0^tG_1(r)dr\right)^{\frac{qc_2}{\beta b_1}-p}\right]\\
&=k_{1,p}(t)\mathbb{E}\left[\left(1+b_1\int_0^tG_1(r)e^{p\sigma_1^2r}dr\right)^{\frac{qc_2}{\beta b_1}-p}\right]\\
&=k_{1,p}(t)\left\|1+b_1\int_0^t G_1(r)e^{p\sigma_1^2r}dr\right\|_{\mathbb{L}^{\frac{qc_2}{\beta b_1}-p}(\Omega)}^{\frac{qc_2}{\beta b_1}-p}.
\end{align*}
Observe that the condition $\frac{qc_2}{\beta b_1}-p\geq 1$ allows for the use of triangle and Minkowski's inequalities for the norm of the space $\mathbb{L}^{\frac{qc_2}{\beta b_1}-p}(\Omega)$; therefore, we obtain
\begin{align*}
\mathcal{I}_1&=k_{1,p}(t)\left\|1+b_1\int_0^t G_1(r)e^{p\sigma_1^2r}dr\right\|_{\mathbb{L}^{\frac{qc_2}{\beta b_1}-p}(\Omega)}^{\frac{qc_2}{\beta b_1}-p}\\
&\leq k_{1,p}(t)\left(1+b_1\left\|\int_0^t G_1(r)e^{p\sigma_1^2r}dr\right\|_{\mathbb{L}^{\frac{qc_2}{\beta b_1}-p}(\Omega)}\right)^{\frac{qc_2}{\beta b_1}-p}\\
&\leq k_{1,p}(t)\left(1+b_1\int_0^t\left\| G_1(r)\right\|_{\mathbb{L}^{\frac{qc_2}{\beta b_1}-p}(\Omega)}e^{p\sigma_1^2r}dr\right)^{\frac{qc_2}{\beta b_1}-p}\\
&= k_{1,p}(t)\left(1+b_1x\frac{e^{\left(a_1+\left(\frac{qc_2}{\beta b_1}+p-1\right)\frac{\sigma_1^2}{2}\right)t}-1}{a_1+\left(\frac{qc_2}{\beta b_1}+p-1\right)\frac{\sigma_1^2}{2}}\right)^{\frac{qc_2}{\beta b_1}-p}.
\end{align*}
Combining the estimates for $\mathcal{I}_1$ and $\mathcal{I}_2$ we obtain
\begin{align*}
\mathbb{E}\left[X(t)^pY(t)^q\right]\leq& 2k_{1,p}(t)k_{2,q}(t)\left(1+b_1x\frac{e^{\left(a_1+\left(\frac{qc_2}{\beta b_1}+p-1\right)\frac{\sigma_1^2}{2}\right)t}-1}{a_1+\left(\frac{qc_2}{\beta b_1}+p-1\right)\frac{\sigma_1^2}{2}}\right)^{\frac{qc_2}{\beta b_1}-p}\\
&\quad\times\int_0^{+\infty}\left(1+b_2y e^{-\sigma_2 z}K_{2,q}(t)\right)^{-q}\mathcal{N}_{0,t}(z)dz. 
\end{align*}
 
\item From (\ref{boundY}) and  (\ref{boundX2}) we can write
\begin{align*}
\mathbb{E}\left[X(t)^pY(t)^q\right]&\geq e^{-pc_1t}\mathbb{E}\left[L_1(t)^pL_2(t)^q\right]\\ 
&=e^{-pc_1t}\mathbb{E}\left[L_1(t)^p\right]\mathbb{E}\left[L_2(t)^q\right].
\end{align*}
Inequality (\ref{logistic lower moment}) completes the proof of (\ref{joint lower moment}).  
 
\item The lower bound (\ref{lower moments Y}) is obtained setting $p=0$ in (\ref{joint lower moment}); to prove the lower bound (\ref{lower moments X}) we observe that
\begin{align}\label{ab}
X(t)&\geq L_1(t)e^{-\frac{c_1}{\beta b_2}\left(1+b_1\int_0^t G_1(v)dv\right)^{\frac{c_2}{\beta b_1}}\ln\left(1+b_2\int_0^t G_2(r)dr\right)}\nonumber\\
&=\frac{G_1(t)e^{-\frac{c_1}{\beta b_2}\left(1+b_1\int_0^t G_1(v)dv\right)^{\frac{c_2}{\beta b_1}}\ln\left(1+b_2\int_0^t G_2(r)dr\right)}}{1+b_1\int_0^t G_1(r)dr}\nonumber\\
&\geq\frac{G_1(t)e^{-\frac{c_1}{\beta b_2}\left(1+b_1K_1(t)e^{\sigma_1M_1(t)}\right)^{\frac{c_2}{\beta b_1}}\ln\left(1+b_2K_2(t)e^{\sigma_2M_2(t)}\right)}}{1+b_1K_1(t)e^{\sigma_1M_1(t)}}\nonumber\\
&=\frac{k_1(t)e^{\sigma_1B_1(t)-\frac{c_1}{\beta b_2}\left(1+b_1K_1(t)e^{\sigma_1M_1(t)}\right)^{\frac{c_2}{\beta b_1}}\ln\left(1+b_2K_2(t)e^{\sigma_2M_2(t)}\right)}}{1+b_1K_1(t)e^{\sigma_1M_1(t)}}.
\end{align}
The last member above is a function of the three dimensional random vector $(B_1(t),M_1(t),M_2(t))$ whose joint probability density function is given by
\begin{align*}
&f_{B_1(t),M_1(t),M_2(t)}(u_1,v_1,v_2)\\
&\quad=\begin{cases}
-4\mathcal{N}_{0,t}'(2v_1-u_1)\mathcal{N}_{0,t}(v_2),&\mbox{ if $v_1>0$, $u_1<v_1$ and $v_2>0$},\\
0,&\mbox{ otherwise}.
\end{cases}
\end{align*}
Therefore, for any $p\geq 0$ we get
\begin{align*}
\mathbb{E}[X(t)^p]&\geq\mathbb{E}\left[\left|\frac{k_1(t)e^{\sigma_1B_1(t)-\frac{c_1}{\beta b_2}\left(1+b_1K_1(t)e^{\sigma_1M_1(t)}\right)^{\frac{c_2}{\beta b_1}}\ln\left(1+b_2K_2(t)e^{\sigma_2M_2(t)}\right)}}{1+b_1K_1(t)e^{\sigma_1M_1(t)}}\right|^p\right]\\
&=-4k_1(t)^p\int_A\frac{e^{p\sigma_1u_1-\frac{pc_1}{\beta b_2}\left(1+b_1K_1(t)e^{\sigma_1v_1}\right)^{\frac{c_2}{\beta b_1}}\ln\left(1+b_2K_2(t)e^{\sigma_2v_2}\right)}}{(1+b_1K_1(t)e^{\sigma_1v_1})^p}\\
&\quad\quad\quad\quad\quad\quad\times\mathcal{N}_{0,t}'(2v_1-u_1)\mathcal{N}_{0,t}(v_2)du_1dv_1dv_2,
\end{align*}
where 
\begin{align*}
A:=\{(u_1,v_1,v_2)\in\mathbb{R}^3: v_1>0,u_1<v_1,v_2>0\}.
\end{align*}
This proves (\ref{lower moments X}).
\end{enumerate}
\end{proof}

\begin{remark}
Due to the complexity of the left hand side in (\ref{boundX}) we were not able to obtain a lower bound for the joint moments $\mathbb{E}\left[X(t)^pY(t)^q\right]$ in the regime $a_1<\phi$. However, according to the argument of Remark \ref{all parameters}, inequality (\ref{joint lower moment}) can be utilize also in that regime.
\end{remark}

\section{Third main theorem: bounds for the distribution functions}

The last main theorem of this paper concerns with upper and lower estimates for the distribution functions of $X(t)$ and $Y(t)$.

\begin{theorem}\label{main theorem d.f.}
Let $\{(X(t),Y(t))\}_{t\geq 0}$ be the unique global strong solution of (\ref{Mao system}). Then, for all $t\geq 0$ and $z_1,z_2>0$ we have the following bounds:
\begin{enumerate}
\item 
\begin{align}\label{lower df X}
\mathbb{P}(X(t)\leq z_1)\geq-2\int_{\left\{\frac{k_1(t)e^{\sigma u}}{1+b_1K_1(t)e^{\sigma_1 v}}\leq z_1\right\}\cap\{v>0\}\cap\{u<v\}}\mathcal{N}_{0,t}'(2v-u)dudv, 
\end{align}	
and 
\begin{align}\label{lower df Y}
&\mathbb{P}(Y(t)\leq z_2)\geq-\frac{4\beta b_1}{\sigma_1 c_2}\int_0^{z_2/(1+b_1K_1(t))^{\frac{c_2}{\beta b_1}}}\left(\int_{\left\{\frac{k_2(t)e^{\sigma_2 u}}{1+b_2K_2(t)e^{\sigma_2 v}}\leq \zeta\right\}\cap\{v<0\}\cap\{u>v\}}\mathcal{N}_{0,t}'(u-2v)dudv\right)\nonumber\\
&\quad\quad\quad\quad\quad\quad\times \mathcal{N}_{0,t}\left(\frac{1}{\sigma_1}\ln\left(\frac{\left(\frac{z}{\zeta}\right)^{\frac{\beta b_1}{c_2}}-1}{b_1K_1(t)}\right)\right)\frac{\left(\frac{z}{\zeta}\right)^{\frac{\beta b_1}{c_2}}}{\left(\frac{z}{\zeta}\right)^{\frac{\beta b_1}{c_2}}-1}\frac{1}{\zeta}d\zeta;
\end{align}	
\item if $a_1>\phi$, then
\begin{align}\label{joint df upper}
\mathbb{P}\left(X(t)\leq z_1,Y(t)\leq z_2\right)&\leq 4\int_{\left\{\frac{k_1(t)e^{\sigma_1 u}}{1+b_1K_1(t)e^{\sigma_1 v}}\leq z_1e^{c_1t}\right\}\cap\{v>0\}\cap\{u<v\}}\mathcal{N}_{0,t}'(2v-u)dudv\nonumber\\
&\quad\times\int_{\left\{\frac{k_2(t)e^{\sigma_2 u}}{1+b_2K_2(t)e^{\sigma_2 v}}\leq z_2\right\}\cap\{v>0\}\cap\{u<v\}}\mathcal{N}_{0,t}'(2v-u)dudv;
\end{align}
\item if $a_1<\phi$, then
\begin{align}\label{upper df X}
\mathbb{P}(X(t)\leq z_1)\leq-4\int_{A_{z_1}\cap\{v_1>0, u_1<v_1, v_2>0\}} \mathcal{N}_{0,t}'(2v_1-u_1)\mathcal{N}_{0,t}(v_2)du_1dv_1dv_2,
\end{align}
where 
\begin{align*}
A_{z_1}:=\left\{(u_1,v_1,v_2)\in\mathbb{R}^3:\frac{k_1(t)e^{\sigma_1u_1-\frac{c_1}{\beta b_2}\left(1+b_1K_1(t)e^{\sigma_1v_1}\right)^{\frac{c_2}{\beta b_1}}\ln\left(1+b_2K_2(t)e^{\sigma_2v_2}\right)}}{1+b_1K_1(t)e^{\sigma_1v_1}}\leq z_1\right\},
\end{align*}
and
\begin{align}\label{upper df Y}
\mathbb{P}(Y(t)\leq z_2)\leq-2\int_{\left\{\frac{k_2(t)e^{\sigma u}}{1+b_2K(t)e^{\sigma_2 v}}\leq z_2\right\}\cap\{v>0\}\cap\{u<v\}}\mathcal{N}_{0,t}'(2v-u)dudv.
\end{align}	
\end{enumerate}
Here,
\begin{align*}
k_1(t):=xe^{(a_1-\sigma_1^2/2)t}\quad\mbox{ and }\quad K_1(t):=x\frac{e^{(a_1-\sigma_1^2/2)t}-1}{a_1-\sigma_1^2/2},
\end{align*}
while
\begin{align*}
k_2(t):=ye^{(a_2-\sigma_2^2/2)t}\quad\mbox{ and }\quad K_2(t):=y\frac{e^{(a_2-\sigma_2^2/2)t}-1}{a_2-\sigma_2^2/2}. 
\end{align*}
\end{theorem}

\begin{proof}
\begin{enumerate}
	\item The upper bound in (\ref{boundX}) (or (\ref{boundX2})) yields
	\begin{align*}
	\mathbb{P}(X(t)\leq z_1)\geq\mathbb{P}(L_1(t)\leq z_1)
	\end{align*}
	which in combination with (\ref{lower df logistic}) gives (\ref{lower df X}). We now prove (\ref{lower df Y}); the estimate
	\begin{align*}
	\int_0^t G_1(r)dr\leq K_1(t)e^{\sigma_1 M_1(t)},
	\end{align*}
	together with the upper estimate in (\ref{boundY}), entails
	\begin{align*}
	\mathbb{P}(Y(t)\leq z_2)&\geq \mathbb{P}\left(L_2(t)\left(1+b_1\int_0^t G_1(r)dr\right)^{\frac{c_2}{\beta b_1}}\leq z_2\right)\\
	&=\mathbb{E}\left[\mathbb{P}\left(L_2(t)\left(1+b_1\int_0^t G_1(r)dr\right)^{\frac{c_2}{\beta b_1}}\leq z_2\bigg|\mathcal{F}_t^2\right)\right]\\
	&=\mathbb{E}\left[\mathbb{P}\left(\left(1+b_1\int_0^t G_1(r)dr\right)^{\frac{c_2}{\beta b_1}}\leq \frac{z_2}{L_2(t)}\bigg|\mathcal{F}_t^2\right)\right]\\
	&=\mathbb{E}\left[\mathbb{P}\left(\int_0^t G_1(r)dr\leq \left(\left(\frac{z_2}{L_2(t)}\right)^{\frac{\beta b_1}{c_2}}-1\right)/b_1\bigg|\mathcal{F}_t^2\right)\right]\\
	&\geq\mathbb{E}\left[\mathbb{P}\left(K_1(t)e^{\sigma_1 M_1(t)}\leq \left(\left(\frac{z_2}{L_2(t)}\right)^{\frac{\beta b_1}{c_2}}-1\right)/b_1\bigg|\mathcal{F}_t^2\right)\right]\\
	&=\mathbb{E}\left[\mathbb{P}\left(M_1(t)\leq \frac{1}{\sigma_1}\ln\left(\frac{\left(\frac{z_2}{L_2(t)}\right)^{\frac{\beta b_1}{c_2}}-1}{b_1K_1(t)}\right)\bigg|\mathcal{F}_t^2\right)\right].
	\end{align*}
	Here $\{\mathcal{F}_t^2\}_{t\geq 0}$ denotes the natural augmented filtration of the Brownian motion $\{B_2(t)\}_{t\geq 0}$. Note that the almost sure positivity of the random variable $M_1(t)$ implies that the probability in the last member above is different from zero if and only if 
	\begin{align*}
	\frac{\left(\frac{z_2}{L_2(t)}\right)^{\frac{\beta b_1}{c_2}}-1}{b_1K_1(t)}>1
	\end{align*}
	which is equivalent to say that
	\begin{align*}
	L_2(t)\leq \frac{z_2}{(1+b_1K_1(t))^{\frac{c_2}{\beta b_1}}}.
	\end{align*}
	Therefore,
	\begin{align*}
	\mathbb{P}(Y(t)\leq z_2)&\geq\mathbb{E}\left[\mathbb{P}\left(M_1(t)\leq \frac{1}{\sigma_1}\ln\left(\frac{\left(\frac{z_2}{L_2(t)}\right)^{\frac{\beta b_1}{c_2}}-1}{b_1K_1(t)}\right)\bigg|\mathcal{F}_t^2\right)\right]\\
	&=\int_0^{z_2/(1+b_1K_1(t))^{\frac{c_2}{\beta b_1}}}\mathbb{P}\left(M_1(t)\leq \frac{1}{\sigma_1}\ln\left(\frac{\left(\frac{z}{\zeta}\right)^{\frac{\beta b_1}{c_2}}-1}{b_1K_1(t)}\right)\right)dF_2(\zeta),
	\end{align*}
	where $F_2$ denotes the distribution function of the random variable $L_2(t)$. We now integrate by parts and notice that $\mathbb{P}\left(M_1(t)\leq \frac{1}{\sigma_1}\ln\left(\frac{\left(\frac{z}{\zeta}\right)^{\frac{\beta b_1}{c_2}}-1}{b_1K_1(t)}\right)\right)=0$ if $\zeta=z_2/(1+b_1K_1(t))^{\frac{c_2}{\beta b_1}}$ while $F_2(\zeta)=0$ when $\zeta=0$. This gives
	\begin{align*}
	\mathbb{P}(Y(t)\leq z_2)&\geq\int_0^{z_2/(1+b_1K_1(t))^{\frac{c_2}{\beta b_1}}}\mathbb{P}\left(M_1(t)\leq \frac{1}{\sigma_1}\ln\left(\frac{\left(\frac{z}{\zeta}\right)^{\frac{\beta b_1}{c_2}}-1}{b_1K_1(t)}\right)\right)dF_2(\zeta)\\
	&=\frac{2\beta b_1}{\sigma_1 c_2}\int_0^{z_2/(1+b_1K_1(t))^{\frac{c_2}{\beta b_1}}}F_2(\zeta)\mathcal{N}_{0,t}\left(\frac{1}{\sigma_1}\ln\left(\frac{\left(\frac{z}{\zeta}\right)^{\frac{\beta b_1}{c_2}}-1}{b_1K_1(t)}\right)\right)\frac{\left(\frac{z}{\zeta}\right)^{\frac{\beta b_1}{c_2}}}{\left(\frac{z}{\zeta}\right)^{\frac{\beta b_1}{c_2}}-1}\frac{1}{\zeta}d\zeta.
	\end{align*}
	Moreover, since from (\ref{lower df logistic}) we know that
	\begin{align*}
	F_2(\zeta)=\mathbb{P}(L_2(t)\leq \zeta)\geq -2\int_{\left\{\frac{k_2(t)e^{\sigma_2 u}}{1+b_2K_2(t)e^{\sigma_2 v}}\leq \zeta\right\}\cap\{v<0\}\cap\{u>v\}}\mathcal{N}_{0,t}'(u-2v)dudv,
	\end{align*}
	we can conclude that
	\begin{align*}
	&\mathbb{P}(Y(t)\leq z_2)\geq\frac{2\beta b_1}{\sigma_1 c_2}\int_0^{z_2/(1+b_1K_1(t))^{\frac{c_2}{\beta b_1}}}F_2(\zeta)\mathcal{N}_{0,t}\left(\frac{1}{\sigma_1}\ln\left(\frac{\left(\frac{z}{\zeta}\right)^{\frac{\beta b_1}{c_2}}-1}{b_1K_1(t)}\right)\right)\frac{\left(\frac{z}{\zeta}\right)^{\frac{\beta b_1}{c_2}}}{\left(\frac{z}{\zeta}\right)^{\frac{\beta b_1}{c_2}}-1}\frac{1}{\zeta}d\zeta\\
	&\quad\geq-\frac{4\beta b_1}{\sigma_1 c_2}\int_0^{z_2/(1+b_1K_1(t))^{\frac{c_2}{\beta b_1}}}\left(\int_{\left\{\frac{k_2(t)e^{\sigma_2 u}}{1+b_2K_2(t)e^{\sigma_2 v}}\leq \zeta\right\}\cap\{v<0\}\cap\{u>v\}}\mathcal{N}_{0,t}'(u-2v)dudv\right)\\
	&\quad\quad\quad\quad\quad\quad\times \mathcal{N}_{0,t}\left(\frac{1}{\sigma_1}\ln\left(\frac{\left(\frac{z}{\zeta}\right)^{\frac{\beta b_1}{c_2}}-1}{b_1K_1(t)}\right)\right)\frac{\left(\frac{z}{\zeta}\right)^{\frac{\beta b_1}{c_2}}}{\left(\frac{z}{\zeta}\right)^{\frac{\beta b_1}{c_2}}-1}\frac{1}{\zeta}d\zeta.
	\end{align*}
	
	\item Using the lower bounds in (\ref{boundY}) and (\ref{boundX2}) we obtain
	\begin{align*}
	\mathbb{P}\left(X(t)\leq z_1,Y(t)\leq z_2\right)&\leq\mathbb{P}\left(L_1(t)e^{-c_1t}\leq z_1, L_2(t)\leq z_2\right)\\
	&=\mathbb{P}\left(L_1(t)e^{-c_1t}\leq z_1\right)\mathbb{P}\left( L_2(t)\leq z_2\right)\\
	&=\mathbb{P}\left(L_1(t)\leq z_1e^{c_1t}\right)\mathbb{P}\left( L_2(t)\leq z_2\right).
	\end{align*}
	With the help of (\ref{upper df logistic}) we conclude that
	\begin{align*}
	\mathbb{P}\left(X(t)\leq z_1,Y(t)\leq z_2\right)&\leq 4\int_{\left\{\frac{k_1(t)e^{\sigma_1 u}}{1+b_1K_1(t)e^{\sigma_1 v}}\leq z_1e^{c_1t}\right\}\cap\{v>0\}\cap\{u<v\}}\mathcal{N}_{0,t}'(2v-u)dudv\\
	&\quad\times\int_{\left\{\frac{k_2(t)e^{\sigma_2 u}}{1+b_2K_2(t)e^{\sigma_2 v}}\leq z_2\right\}\cap\{v>0\}\cap\{u<v\}}\mathcal{N}_{0,t}'(2v-u)dudv
	\end{align*}
	
	\item We now prove (\ref{upper df X}); we know from (\ref{boundX}) and (\ref{ab}) that
\begin{align*}
X(t)&\geq L_1(t)e^{-\frac{c_1}{\beta b_2}\left(1+b_1\int_0^t G_1(v)dv\right)^{\frac{c_2}{\beta b_1}}\ln\left(1+b_2\int_0^t G_2(r)dr\right)}\\
&\geq\frac{k_1(t)e^{\sigma_1B_1(t)-\frac{c_1}{\beta b_2}\left(1+b_1K_1(t)e^{\sigma_1M_1(t)}\right)^{\frac{c_2}{\beta b_1}}\ln\left(1+b_2K_2(t)e^{\sigma_2M_2(t)}\right)}}{1+b_1K_1(t)e^{\sigma_1M_1(t)}}.
\end{align*}
Hence, we can write
\begin{align*}
\mathbb{P}(X(t)\leq z_1)&\leq\mathbb{P}\left(\frac{k_1(t)e^{\sigma_1B_1(t)-\frac{c_1}{\beta b_2}\left(1+b_1K_1(t)e^{\sigma_1M_1(t)}\right)^{\frac{c_2}{\beta b_1}}\ln\left(1+b_2K_2(t)e^{\sigma_2M_2(t)}\right)}}{1+b_1K_1(t)e^{\sigma_1M_1(t)}}\leq z_1\right)\\ 
&=-4\int_{A_{z_1}\cap\{v_1>0, u_1<v_1, v_2>0\}} \mathcal{N}_{0,t}'(2v_1-u_1)\mathcal{N}_{0,t}(v_2)du_1dv_1dv_2,
\end{align*}
where 
\begin{align*}
A_{z_1}:=\left\{(u_1,v_1,v_2)\in\mathbb{R}^3:\frac{k_1(t)e^{\sigma_1u_1-\frac{c_1}{\beta b_2}\left(1+b_1K_1(t)e^{\sigma_1v_1}\right)^{\frac{c_2}{\beta b_1}}\ln\left(1+b_2K_2(t)e^{\sigma_2v_2}\right)}}{1+b_1K_1(t)e^{\sigma_1v_1}}\leq z_1\right\}.
\end{align*}
This coincides with (\ref{upper df X}). Moreover, from the lower estimate in (\ref{boundY}) we get
\begin{align*}
\mathbb{P}(Y(t)\leq z_2)\leq \mathbb{P}(L_2(t)\leq z_2);
\end{align*}
inequality (\ref{upper df logistic}) completes the proof of (\ref{upper df Y}). 
\end{enumerate}
\end{proof}

\section{Discussion}

In this paper, we propose a finite-time analysis for the solution of the two dimensional system (\ref{Mao system}) which describes a foraging arena model in presence of environmental noise. We derive in Theorem \ref{a.s. bounds} almost sure upper and lower bounds for the components on the solution vector; these bounds emphasis the interplay between the parameters describing the model and different sources of randomness involved in the system. While such relationship is hardly visible in the description of the asymptotic behaviour of the solution, our estimates agree, if let the time tend to infinity, with the classification in asymptotic regimes obtained by \cite{Mao foraging}: this is shown in details in Remark \ref{comparison with asymptotic}. The accuracy of our bounds, which are obtained via a careful use of comparison theorems for stochastic differential equations, is evident in the simulations below (see Figure \ref{plots}). There we plot for a given set of parameters the solution of the deterministic version of (\ref{Mao system}), i.e. with $\sigma_1=\sigma_2=0$, a computer simulation of the solution of the stochastic equation (\ref{Mao system}) for different noise intensities and the corresponding upper and lower bounds from Theorem \ref{a.s. bounds}.\\
Then, we utilize the bounds for the solution from Theorem \ref{a.s. bounds} to derive two sided estimates for some statistical aspects of the solution. More precisely, in Theorem \ref{main theorem moments} and Theorem \ref{main theorem d.f.} we propose upper and lower bounds for the joint moments and distribution function of the components of the solution vector, respectively. These estimates are expressed via integrals whose numerical approximation is pretty standard. Again, the roles of the parameters describing our model are explicitly described in the proposed estimates.

\begin{figure}[H]
	\centering
	\includegraphics[width=1.0\linewidth]{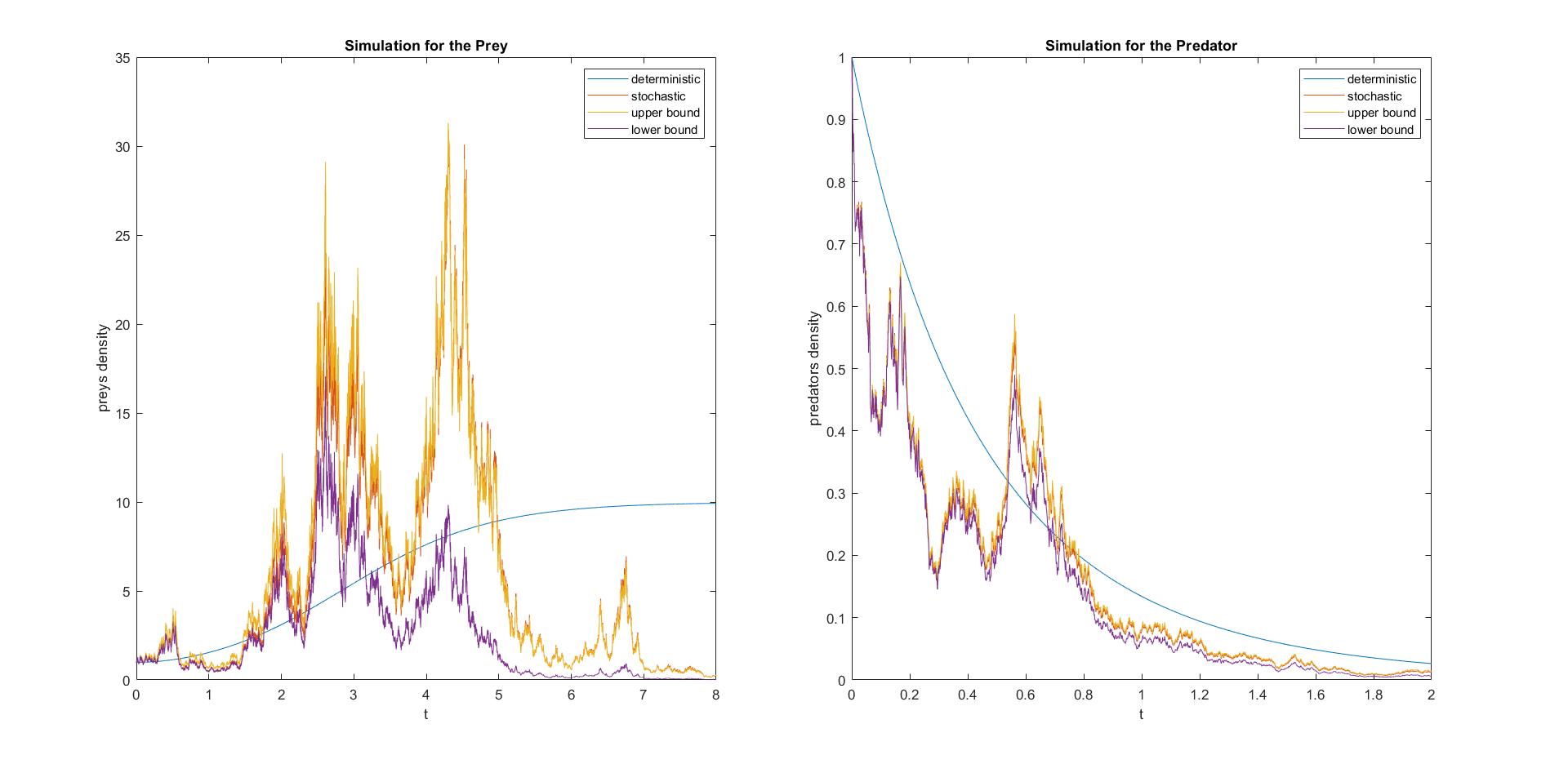}
	\includegraphics[width=1.0\linewidth]{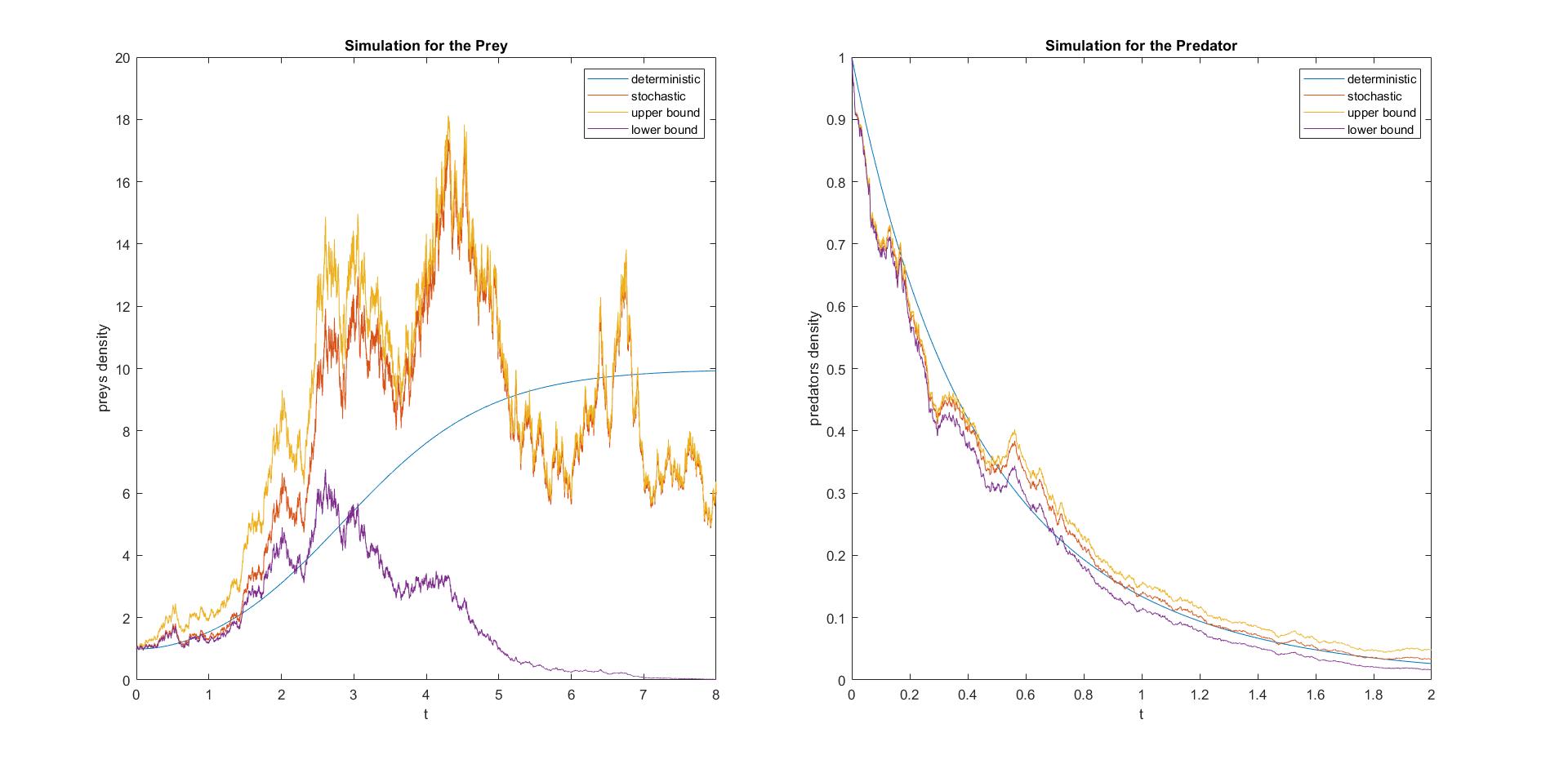}
	\caption{Comparing the paths of $X(t)$ (prey) and $Y(t)$ (predator) with the corresponding upper and lower bounds from Theorem \ref{a.s. bounds} for system \ref{Mao system} with $a_1=1$, $b_1=0.1$, $c_1=6$, $a_2=2$, $b_2=0.5$, $c_2=0.9$ and $\beta=5$ under different noise intensity: $\sigma_1=1.5$, $\sigma_2=1.3$ (top figures) and $\sigma_1=0.5$, $\sigma_2=0.3$ (bottom figures)}
	\label{plots}
\end{figure}

\end{document}